\documentclass[a4paper]{article}
\usepackage{latexsym,bm}
\usepackage{amssymb,amsmath,amsthm}
\usepackage[english,german]{babel}
\usepackage[colorlinks=true]{hyperref}
\usepackage{color}
\allowdisplaybreaks[4]
\usepackage[all]{xy}
\newtheorem{theorem}{Theorem}[section]
\newtheorem{lemma}[theorem]{Lemma}
\newtheorem{remark}[theorem]{Remark}
\newtheorem{proposition}[theorem]{Proposition}
\newtheorem{definition}[theorem]{Definition}
\newtheorem{corollary}[theorem]{Corollary}
\numberwithin{equation}{section}
\hypersetup{linkcolor=blue,urlcolor=red,citecolor=red}
\title{Stabilizability of linear systems with discrete observation mode\thanks{This work was partially supported by the National Natural Science Foundation of China under grant 11971022, and by Fundamental Research Funds for the Central Universities, China University of Geosciences (Wuhan) (CUGSX01)}}
\author{Hanbing Liu\thanks{School of Mathematics and Physics, China University of Geosciences (Wuhan), Wuhan, 430074, China (hanbing272003@aliyun.com)} \and Gengsheng Wang\thanks{Center for Applied Mathematics, Tianjin University, Tianjin, 300072, China (wanggs@yeah.net)}  \and Huaiqiang Yu\thanks{School of Mathematics, Tianjin University, Tianjin 300354, China (huaiqiangyu@tju.edu.cn,  huaiqiangyu@yeah.net)}}

\date{}
\begin{document}
\selectlanguage{english}
\maketitle
\begin{abstract}
 For linear control systems, the usual state feedback stabilizability has two  components: one is a continuous observation mode (i.e., to observe solutions continuously in time), and the other is a class of feedback laws (which is usually the space of all of the linear and bounded operators from a state space to a control space). This paper
studies the stabilizability for abstract linear control systems, with a discrete observation
mode (i.e., to  observe solutions discretely in time) and two different classes of feedback laws.
 We first characterize these types of stabilizabilities via some weak observability inequalities for the dual systems. Then, we use these characterizations to reveal the connections between
  these types of stabilizabilities and those with continuous observation mode. Finally, we show some applications of the aforementioned weak observability inequalities.

\end{abstract}

{\bf Keywords}. stabilizability, observation modes, classes of feedback laws, weak observability inequalities
\vskip 5pt
{\bf AMS subject classifications.} 93D15, 93D23, 93C57, 93C25

\section{Introduction}\label{yu-section-1}
 \subsection{Notation}
   We write $\mathbb{R}^+:=(0,+\infty)$; $\mathbb{N}:=\{0, 1,2,\ldots\}$; $\mathbb{N}^+:=\{1,2,\ldots\}$;
    $\mathbb{Q}^+$ for the set of all positive rational numbers.
       Given  $t\in\mathbb{R}^+$, we let $[t]:=\max\{n\in\mathbb{N}:n\leq t\}$.
   Given a Hilbert space $Y$, we write $\|\cdot\|_{Y}$ and $\langle\cdot,\cdot\rangle_{Y}$
   for its norm and  inner product, respectively, and we identify it with its dual space. We write $C(\mathbb{R}^+;Y)$
   for the space of all of the continuous functions from $\mathbb{R}^+$ to $Y$ and $PC(\mathbb{R}^+;Y)$
   for the space of all of the step functions from $\mathbb{R}^+$ to $Y$. When $L$ is a densely defined and closed linear operator on a Hilbert space $Y$, we use $D(L)(=\{y\in Y: Ly\in Y\})$ to denote its domain
   (which is a Hilbert space with the norm $\|y\|_{D(L)}:=(\|y\|_{Y}^2+\|Ly\|_{Y}^2)^{1/2}$). We use
   $\rho(L)$ to denote its resolvent set and use $\sigma(L)$ to denote the complementary set of  $\rho(L)$.
     Given two Hilbert spaces $Y_1$ and $Y_2$, we write $\mathcal{L}(Y_1;Y_2)$ for the space of all of the linear and bounded operators from $Y_1$ to $Y_2$ and let $\mathcal{L}(Y_1):=\mathcal{L}(Y_1;Y_1)$. Given $F\in\mathcal{L}(Y)$, we write $F^*\in\mathcal{L}(Y)$ for its adjoint operator. We denote by
     $I$ the identity operator on any space.
     We  use $\verb"i"$ to denote the unitary imaginary number, i.e., $\verb"i"^2=-1$. We use $\mathcal{F}$ and $\mathcal{F}^{-1}$ to denote the Fourier transform and the inverse of the Fourier transform, respectively. Given a matrix $P$, we
    use  $P^\top$ to denote its transposition.
          For any set $E$, we let $\chi_E$ be its characteristic function. We use
$C(\cdots)$ to denote a positive constant that depends on what is enclosed in the brackets.

\subsection{Stabilizability, observation modes, and classes of feedback laws}
    Let $Y$ and $U$ be two Hilbert spaces.
   Furthermore, let $A: D(A)\subset Y\to Y$ generate a $C_0$-semigroup
 $\{S(t)\}_{t\geq 0}$, and let $B\in\mathcal{L}(U;Y)$.
        We consider the following control system:
\begin{equation}\label{yu-5-26-1}
    y'(t)=Ay(t)+Bu(t),\;\;t\in\mathbb{R}^+,
\end{equation}
    where $u\in L^2(\mathbb{R}^+;U)$.
   We write $y(\cdot;y_0,u)$, with $y_0\in Y$, for
     the solution of the system (\ref{yu-5-26-1}) with the initial condition
    $y(0)=y_0$.
     This paper aims to study several kinds of exponential stabilizabilities (stabilizability, for short) for the system \eqref{yu-5-26-1}.
\par
 The usual stabilizability task is to find a feedback law $F\in \mathcal{L}(Y;U)$ such that
the following closed-loop system is exponentially stable:
\begin{equation}\label{yu-6-18-17-12}
    y'(t)=(A+BF)y(t),\;\;t\in\mathbb{R}^+.
\end{equation}
\begin{remark}\label{remark1.1w}
Several notes on the  stabilizability above are stated as follows.
\begin{itemize}
  \item [(a1)]  This is a kind of state feedback  stabilizability.
  Its two key components are as follows: the first is the observation mode to observe a solution $y(\cdot)$  continuously in time over
  $\mathbb{R}^+$ (i.e., the observation system reads $z(t)=y(t), t\in\mathbb{R}^+$); the second is the class of feedback laws $\mathcal{L}(Y;U)$.
   This kind of stabilizability has been widely studied
   (e.g., \cite{Liu-22, Ma-22, Pritchard, Trelat-20} and the references therein).

  \item [(a2)]  To distinguish this kind of stabilizability from others, we call it the \textbf{(CC)}-stabilizability (i.e., the stabilizability
   with the continuous observation mode and feedback laws in $\mathcal{L}(Y;U)$).
  Notice that the first and the second \textbf{C} in the \textbf{(CC)}-stabilizability
  represent ``continuous observation mode'' and``constant-valued feedback laws,'' respectively.
  \item [(a3)]
 The \textbf{(CC)}-stabilizability for the system (\ref{yu-5-26-1})
  is characterized  by the weak observability inequality
 \begin{equation}\label{yu-6-18-3}
    \|S(T)^*\varphi\|_Y^2\leq C\int_0^T\|B^*S(T-t)^*\varphi\|_{U}^2dt+\delta\|\varphi\|^2_Y\;\;\mbox{for any}\;\;\varphi\in Y,
\end{equation}
  where  $T>0$,  $\delta\in(0,1)$, and
    $C\geq 0$ are constants independent of $\varphi$.
    (see \cite[Theorem 1]{Trelat-20} and  also \cite[Section 3.1]{Liu-22} for the case where $B$ is allowed to be unbounded.)

\end{itemize}
\end{remark}

We now introduce another kind of observation mode.
We observe a solution $y(\cdot)$ (of the system (\ref{yu-5-26-1}))
such that
\begin{equation*}\label{yu-6-25-1}
    z(t)=(H_Ty)(t),
\end{equation*}
where $H_T$ (with $T>0$ arbitrarily fixed) is the linear operator from
 $C(\mathbb{R}^+;Y)$ to $PC(\mathbb{R}^+;Y)$, defined as follows: For each  $f\in C(\mathbb{R}^+;Y)$,
 \begin{equation}\label{yu-6-25-2}
    (H_Tf)(t):= f([t/T]T),\;\;
   \;t\in\mathbb{R}^+.
\end{equation}
This is a kind of discrete (in time) observation mode. The observation (or sampling) occurs at the moments: $0$, $T$, $2T,\dots$.
$z(\cdot)$ is a special $Y$-valued step function that takes only one value in each time interval $[kT, (k+1)T)$
with $k\in \mathbb{N}$.
{\it We call this observation mode the $T$-periodic discrete observation mode, and
 $T$ is called an observation (or sampling) period.} We treat
 $H_T$
as an observation operator in time. This kind of observation mode is more suitable
in practical applications, which elicits another kind of stabilizability
with the following definition:

\begin{definition}\label{yu-definition-5-27-1}
    Let $T>0$ and $H_T$ be given by \eqref{yu-6-25-2}. The system (\ref{yu-5-26-1})  is said to be \textbf{(DC)}$_T$-stabilizable
    (i.e.,
     stabilizable with the $T$-periodic discrete observation mode
     and  feedback laws in $\mathcal{L}(Y;U)$), if there is an $F\in
     \mathcal{L}(Y;U)$ such that the following closed-loop system is exponentially stable:
     \begin{equation}\label{yu-6-18-2}
    y'(t)=Ay(t)+BF(H_Ty)(t),\;\;t\in\mathbb{R}^+.
\end{equation}
\end{definition}

\begin{remark}\label{yu-remark-6-18-1}
 Several notes on Definition \ref{yu-definition-5-27-1} are stated as follows:
\begin{itemize}
\item [(b1)] One can directly check that for each $y_0\in Y$, the system \eqref{yu-6-18-2} with the initial condition $y(0)=y_0$ has a unique solution in $C([0,+\infty);Y)$.
  \item [(b2)] In the \textbf{(DC)}$_T$-stabilizability, the observation mode
is to observe solutions $T$-periodic discretely in time (which differs from the observation mode in
the
\textbf{(CC)}-stabilizability), while the class of feedback laws is $\mathcal{L}(Y;U)$ (which is the same as that in the \textbf{(CC)}-stabilizability).
Notice that \textbf{D}, \textbf{C}, and $T$ in the \textbf{(DC)}$_T$-stabilizability
  represent ``discrete observation mode,'' ``constant-valued feedback laws,'' and ``$T$-periodic,'' respectively.
   \item [(b3)] When the system (\ref{yu-5-26-1}) is \textbf{(DC)}$_T$-stabilizable, we write
   $y_F(\cdot)$
   for a solution to \eqref{yu-6-18-2} and then let
   \begin{equation*}
   u(t):=F(H_Ty_F)(t)=\sum_{i=1}^{\infty}\chi_{[(i-1)T, iT)}(t)u_i,\;\; t\in \mathbb{R}^+,
   \end{equation*}
   where $u_i=Fy_F((i-1)T)$ with $i\in \mathbb{N}^+$. Since $y_F$ undergoes exponential decay, we have $(u_i)_{i\in\mathbb{N}^+}\in l^2(\mathbb{N}^+;U)$. Moreover, $y_F(\cdot)$ satisfies
    \begin{equation}\label{e102}
y'(t)=Ay(t)+B\sum_{i=1}^{\infty}\chi_{[(i-1)T, iT)}(t)u_i,\;\;t\in\mathbb{R}^+.
\end{equation}
The system \eqref{e102}, with $(u_i)_{i\in\mathbb{N}^+}\in l^2(\mathbb{N}^+;U)$, can be treated as a
sampled-data control system, which is indeed the system (\ref{yu-5-26-1}) with controls in the following
subspace:
\begin{equation*}
\Big\{u\in L^2(\mathbb{R}^+;U)\;|\; u(t)=\sum_{i=1}^{\infty}\chi_{[(i-1)T, iT)}(t)u_i,\; (u_i)_{i\in\mathbb{N}^+}\in l^2(\mathbb{N}^+;U)\Big\}.
\end{equation*}
 It deserves mentioning that sampled-data control systems have been
widely studied (e.g. \cite{Chen-96, Kabamba-87, Logemann-05, Rebarber-97, Rebarber-98, Rosen-92} and the references therein).

  \end{itemize}
    \end{remark}

We next define another class of feedback laws. For each
  $T>0$, we let
   \begin{equation}\label{yu-6-21-2}
    \mathcal{F}_{T,\infty}:=\{F(\cdot)\in L^{\infty}_{loc}(\mathbb{R}^+; \mathcal{L}(Y; U)): F(t+T)=F(t) \;\mbox{a.e.}\;t\in\mathbb{R}^+\}.
\end{equation}
   This new class of feedback laws leads to the following new kinds of stabilizability for the system (\ref{yu-5-26-1}).
\begin{definition}\label{yu-def-6-21-1}
    Let $T>0$, and let $H_T$ and $\mathcal{F}_{T,\infty}$ be given by
         (\ref{yu-6-25-2}) and \eqref{yu-6-21-2}, respectively.
\begin{enumerate}
  \item [$(i)$] The system (\ref{yu-5-26-1}) is said to be \textbf{(DP)}$_{T}$-stabilizable (i.e.,
  stabilizable
    with the $T$-periodic discrete observation mode and feedback laws in $\mathcal{F}_{T,\infty}$), if there is an $F(\cdot)\in \mathcal{F}_{T,\infty}$ such that the following system is exponentially stable:
\begin{equation}\label{yu-6-21-1}
    y'(t)=Ay(t)+BF(t)(H_Ty)(t),\;\;t\in\mathbb{R}^+.
\end{equation}
  \item [$(ii)$]  The system (\ref{yu-5-26-1}) is said to be \textbf{(CP)}$_{T}$-stabilizable (i.e.,
  stabilizable
    with the continuous observation mode and feedback laws in $\mathcal{F}_{T,\infty}$) if there is
    an $F(\cdot)\in \mathcal{F}_{T,\infty}$ such that the following system is exponentially stable:
\begin{equation}\label{yu-6-28-1wang7-13}
    y'(t)=Ay(t)+BF(t)y(t),\;\;t\in\mathbb{R}^+.
\end{equation}
\end{enumerate}
\end{definition}
\begin{remark}
Several notes on Definition \ref{yu-def-6-21-1} are as follows:
\begin{itemize}
 \item [(c1)] One can directly check that for each $y_0\in Y$, the system \eqref{yu-6-21-1}/\eqref{yu-6-28-1wang7-13} with the initial condition $y(0)=y_0$ has a unique solution in $C([0,+\infty);Y)$.
  \item [(c2)] In the \textbf{(DP)}$_{T}$-stabilizability, the observation mode is
  to observe solutions $T$-periodic discretely in time,
  while
  the class of feedback laws is $\mathcal{F}_{T,\infty}$, whose elements are time $T$-periodic
  operator-valued functions. Notice that \textbf{D}, \textbf{P}, and $T$ in the \textbf{(DP)}$_T$-stabilizability
  represent  ``discrete observation mode,'' ``periodic feedback laws,'' and ``$T$-periodic,'' respectively.

  \item [(c3)]
  In  the \textbf{(CP)}$_{T}$-stabilizability,  the observation mode is
  to observe solutions continuously in time, while
  the class of feedback laws is $\mathcal{F}_{T,\infty}$.

  \item [(c4)]  In the \textbf{(DP)}$_{T}$-stabilizability, $T$ plays two roles. One is in the
  observation mode, where the observation occurs at the moments $0$, $T$, $2T,\dots$, while another is in the class of feedback laws, where each $F\in\mathcal{F}_{T,\infty}$ is $T$-periodic. We do not know if taking different values of $T$ in the observation mode and the class of feedback laws would make an essential difference. This might be an interesting problem.

  \item [(c5)] The class $\mathcal{F}_{T,\infty}$ is an expansion of $\mathcal{L}(Y; U)$.
  \end{itemize}
\end{remark}

  \subsection{Motivation and aim}

For the system (\ref{yu-5-26-1}), we have introduced four types of stabilizability, which are the  \textbf{(CC)}-stabilizability,
\textbf{(CP)}$_T$-stabilizability, \textbf{(DC)}$_T$-stabilizability, and \textbf{(DP)}$_T$-stabilizability, where two different observation modes and two different
classes of feedback laws are involved.
  First, as mentioned above, the \textbf{(CC)}-stabilizability is characterized by
the weak observability inequality \eqref{yu-6-18-3}. However,
 the characterizations of the stabilizability for the system (\ref{yu-5-26-1}) with  discrete observation modes
 have not been examined to the best of our knowledge. This motivated us to characterize the \textbf{(DC)}$_T$-stabilizability and the \textbf{(DP)}$_T$-stabilizability via some weak observability inequalities (which may be different from \eqref{yu-6-18-3}).
Second, it is important to illustrate
the difference between the discrete and continuous observation modes.
It would also be interesting to reveal the possibly different influences of different classes of feedback
laws on the stabilizability. This inspired us to find relationships between the above-mentioned
four types of stabilizabilities.

\subsection{Main results}

 The first main result characterizes the \textbf{(DC)}$_T$-stabilizability via some weak observability inequalities.
\begin{theorem}\label{th03}
    Let $T>0$. Then, the following statements are equivalent:
\begin{enumerate}
  \item [(i)] The system (\ref{yu-5-26-1}) is \textbf{(DC)}$_T$-stabilizable.
 \item [(ii)] There are constants $\omega>0$, $C_1\geq 0$, and $C_2>0$ such that
  when $N\in\mathbb{N}^+$,
\begin{equation}\label{e107-b}
\|S(NT)^*\varphi\|_Y^2\leq C_1\sum_{i=1}^{N}\Big\|\int_{(i-1)T}^{iT}B^*S(t)^*\varphi dt\Big\|_U^2+C_2e^{-\omega NT} \|\varphi\|_Y^2\;\;\mbox{for any}\;\;\varphi\in Y.
\end{equation}
      \item [(iii)] There are constants $N\in \mathbb{N}^+$, $\delta\in (0,1)$, and $C\geq 0$ such that
\begin{equation}\label{e107}
\|S(NT)^*\varphi\|_Y^2\leq C\sum_{i=1}^{N}\Big\|\int_{(i-1)T}^{iT}B^*S(t)^*\varphi dt\Big\|_U^2+\delta \|\varphi\|^2_Y\;\;\mbox{for any}\;\;\varphi\in Y.
\end{equation}
   \end{enumerate}
\end{theorem}
%\color{blue} We see from examples in Section  \ref{yu-sec-6-22} that the  concepts of the  \textbf{(DC)}$_T$-stabilizabilty for some $T>0$ and  the \textbf{(DC)}$_T$-stabilizabilty for all $T>0$ are different. For the latter one, we have the following characterization, which is an immediately consequence of Theorem \ref{th03}.
A direct consequence of Theorem \ref{th03} is as follows.
\begin{corollary}\label{coro1.7,7-25}
The system (\ref{yu-5-26-1}) is \textbf{(DC)}$_T$-stabilizable for all $T>0$ if and only if
for each $T>0$, there are constants $N\in\mathbb{N}^+$, $\delta\in(0,1)$, and  $C\geq 0$ such that
    the inequality (\ref{e107}) is true.
\end{corollary}
\color{black}

  The second main result concerns the characterization of the \textbf{(DP)}$_{T}$-stabilizability
  via some weak observability inequalities.
\begin{theorem}\label{th03-1}
The following statements are equivalent:
\begin{enumerate}

  \item[(i)] The system (\ref{yu-5-26-1}) is \textbf{(DP)}$_{T}$-stabilizable for all
  $T>0$.
  \item[(ii)] The system (\ref{yu-5-26-1}) is \textbf{(DP)}$_{T}$-stabilizable
  for some $T>0$.

  \item[(iii)] There are constants $T>0$, $\delta\in(0,1)$, and $C\geq 0$ such that
  the weak observability inequality \eqref{yu-6-18-3} holds.

 \end{enumerate}
\end{theorem}

The last main result gives the relationships between the four types of  stabilizabilities.
%asserts the following relations between the various types of stabilizability,
% \begin{displaymath}    \xymatrix@=8ex{ \textbf{(CC)} \ar@{=>}@<.9ex>[r]\ar@{=>}|\setminus@<.9ex>[r] & \textbf{(DC)}_T \ar@{=>}@<.9ex>[l]\ar@{=>}@<.9ex>[d]  \\   \textbf{(CP)}_T \ar@{<=>}[u] \ar@{<=>}[r] & \textbf{(DP)}_T  \ar@{=>}@<.9ex>[u]\ar@{=>}|\backslash @<.9ex>[u] \ar@{<=>}[ul] }\end{displaymath}
% which reveals not only the differences between the continuous observation mode and the discrete observation mode, but also the meaning of expanding the feedback class. The precise mathematical meaning of the above graph is provided as follows.

\begin{theorem}\label{thm1.11-7-26}
For the system (\ref{yu-5-26-1}), the following conclusions are true:
\begin{enumerate}
\item[(i)] The \textbf{(DC)}$_T$-stabilizability for some $T>0$ implies the \textbf{(CC)}-stabilizability, but the reverse is not true in general.
 \item[(ii)] Both  the \textbf{(DP)}$_T$-stabilizability for some $T>0$ and  the \textbf{(DP)}$_T$-stabilizability for all $T>0$ are equivalent to the \textbf{(CC)}-stabilizability.
\item[(iii)]  Both the \textbf{(CP)}$_T$-stabilizability for some $T>0$ and
the \textbf{(CP)}$_T$-stabilizability for all $T>0$
are equivalent to the  \textbf{(CC)}-stabilizability.
\item[(iv)] The \textbf{(DC)}$_T$-stabilizability for some $T>0$ implies the \textbf{(DP)}$_T$-stabilizability for all $T>0$, but the reverse is not true in general.
  \item[(v)]   The \textbf{(DP)}$_T$-stabilizability for all/some $T>0$ is equivalent to the \textbf{(CP)}$_T$-stabilizability for all/some $T>0$.
\end{enumerate}
\end{theorem}

\begin{remark}\label{remarkwang7-26}
    Several notes on the above theorems are stated as follows.
\begin{enumerate}
\item [(d1)] Theorem \ref{th03} provides a quantitative discriminant criterion
for the \textbf{(DC)}$_T$-stabilizability for some $T>0$. Since
 the \textbf{(DC)}$_T$-stabilizability for some $T>0$ differs from the \textbf{(DC)}$_T$-stabilizability for all $T>0$ (see Example 1 in Subsection \ref{yu-sec-6-22-1}), we present Corollary \ref{coro1.7,7-25}
 as a complement of Theorem \ref{th03}. The inequality (\ref{e107}), combined with (\ref{yu-6-18-3}), can help us to analyze the connections between the \textbf{(DC)}$_T$-stabilizability, \textbf{(DP)}$_T$-stabilizability, and \textbf{(CC)}-stabilizability quantitatively.

  \item [(d2)]
       Theorem \ref{th03-1} states that the \textbf{(DP)}$_T$-stabilizability for some $T>0$ and the \textbf{(DP)}$_T$-stabilizability for all $T>0$ are equivalent, and both of them are characterized by the inequality (\ref{yu-6-18-3}), which
    is
    exactly the weak observability inequality characterizing
      the \textbf{(CC)}-stabilizability.

\item [(d3)]  Theorem \ref{thm1.11-7-26} can be described  by the following diagram:
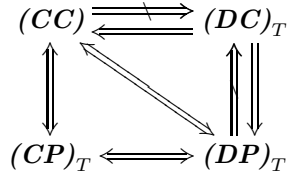
\begin{displaymath}    \xymatrix@=8ex{ \textbf{(CC)} \ar@{=>}@<.9ex>[r]\ar@{=>}|\setminus@<.9ex>[r] & \textbf{(DC)}_T \ar@{=>}@<.9ex>[l]\ar@{=>}@<.9ex>[d]  \\   \textbf{(CP)}_T \ar@{<=>}[u] \ar@{<=>}[r] & \textbf{(DP)}_T  \ar@{=>}@<.9ex>[u]\ar@{=>}|\backslash @<.9ex>[u] \ar@{<=>}[ul] }\end{displaymath}
 Theorem \ref{thm1.11-7-26} provides the following two facts about the stabilizability for the system (\ref{yu-5-26-1}):

\begin{itemize}
  \item The function of the  $T$-periodic
  discrete observation mode is weaker than that of the continuous observation mode from two aspects. First, for the feedback class $\mathcal{L}(Y;U)$, the stabilizability may be lost when the observation mode is changed from the continuous type to the $T$-periodic discrete type
      (see  $(i)$ of Theorem \ref{thm1.11-7-26}). Second, for the continuous observation mode,
      the system (\ref{yu-5-26-1})
      is stabilizable with the class $\mathcal{L}(Y;U)$ if and only if
      it is stabilizable with the class
       $\mathcal{F}_{T,\infty}$, while for the $T$-periodic discrete observation mode,
        when the system (\ref{yu-5-26-1}) is stabilizable with the class
       $\mathcal{F}_{T,\infty}$, it may be not stabilizable with the class $\mathcal{L}(Y;U)$ (see $(iii)$ and $(iv)$
       of Theorem \ref{thm1.11-7-26}).

  \item The expansion of the feedback class can make up for the deficiency of the discrete
  observation mode.
  Indeed,
   the \textbf{(CC)}-stabilizability does not imply the \textbf{(DC)}$_T$-stabilizability for some $T>0$ in general (see $(i)$ of Theorem \ref{thm1.11-7-26}), which shows a deficiency of the discrete
  observation mode relative to the continuous observation mode.
      However, this deficiency is exactly compensated by expanding
      $\mathcal{L}(Y;U)$ to $\mathcal{F}_{T,\infty}$ (see $(ii)$ of Theorem \ref{thm1.11-7-26}).

\end{itemize}

\end{enumerate}
\end{remark}

\subsection{Related work }

Some related work as is mentioned below:

\begin{itemize}

  \item In a finite-dimensional setting where $A$ and $B$ are matrices, some characterizations for
  the \textbf{(DC)}$_T$-stabilizability have been obtained by the Kalman controllability decomposition
  (e.g., \cite{Chen-96, Hutus-1970})

  \item In the infinite-dimensional setting, some characterizations of the \textbf{(DC)}$_T$-stabilizability
 were obtained \cite{Rosen-92, Rebarber-97}
  for the system \eqref{yu-5-26-1}, where some additional assumptions were imposed, such as
  the assumption that the system can be decomposed to an unstable finite dimensional part and a stable infinite dimensional part
  (\cite{Rosen-92}) and that $B$ is compact from $U$ to the dual space of $D(A^*)$  with the pivot space $Y$ (\cite{Rebarber-97}).

  \item For studies on the \textbf{(DP)}$_T$-stabilizability,
   various other studies have been published
       \cite{Kabamba-87, Logemann-05, Rebarber-98, Tarn-88}.

  \item In other related work \cite{Liu-22, Trelat-20},
  some characterizations for the  \textbf{(CC)}-stabilizability were obtained via weak observability
inequalities. The design of feedback laws for the  \textbf{(CC)}-stabilizability based on the weak observability
inequality and a generalized Gramian were also studied \cite{ Ma-22}.
\end{itemize}
\subsection{Organization of this paper}
The rest of this paper is organized as follows. Section 2 presents some preliminary results. Section 3 proves our main theorems. Section 4 gives applications of the main theorems.
 \section{Preliminary}
 This section presents some results on a discrete LQ problem that will be used in the proofs of our main theorems.
{\it Throughout this section, $T>0$ is arbitrarily fixed.} We define two spaces as follows:
\begin{equation}\label{yu-6-21-b-1}
\begin{cases}
    L^2_T(0, NT; U):=\{f\in L^2(0, NT; U): f(\cdot)=\sum_{i=1}^{N}\chi_{[(i-1)T, iT)}(\cdot)f_i, \; \{f_i\}_{i=1}^{N}\subset U\},\;\;N\in\mathbb{N}^+;\\
    L^2_T(\mathbb{R}^+; U):=\{f\in L^2(\mathbb{R}^+; U): f(\cdot)=\sum_{i=1}^{\infty}\chi_{[(i-1)T, iT)}(\cdot)f_i,\; \{f_i\}_{i\in\mathbb{N}^+}\subset U\},
\end{cases}
\end{equation}
    with the $L^2(0,T;U)$-norm and the $L^2(\mathbb{R}^+; U)$-norm, respectively.
  It is clear that $L^2_T(\mathbb{R}^+; U)$/ $L^2_T(0,NT; U)$
    is a closed subspace of $L^2(\mathbb{R}^+;U)$/$L^2(0,NT;U)$ and is isomorphic to $l^2(\mathbb{N}^+;U)$/$U^{N}$.

\par
We are now in a position to introduce the above-mentioned discrete LQ problem.
Let
\begin{equation}\label{2.2wang7-11}
\Phi:=S(T)\;\;\mbox{and}\;\;D:=\int_0^TS(T-t)dtB.
\end{equation}
 Consider the discrete control system
\begin{equation}\label{e202}
y_{i}=\Phi y_{i-1}+Du_i,\;\;i\in\mathbb{N}^+,
\end{equation}
where  $(u_i)_{i\in\mathbb{N}^+}\in l^2(\mathbb{N}^+;U)$.
Two facts are stated. First, given $u_d:=(u_i)_{i\in\mathbb{N}^+}\in l^2(\mathbb{N}^+;U)$ and $y_0\in Y$, the system  \eqref{e202} with the control $u_d$ and the initial datum $y_0$ has a unique solution $(y_i(y_0, u_d))_{i\in\mathbb{N}^+}$. Second,
the connection between the systems  (\ref{e102}) and \eqref{e202} is as follows.
If $y(\cdot)$ is a solution to the system (\ref{e102}) corresponding to the control $u_d:=(u_i)_{i\in\mathbb{N}^+}$, then
$(y(iT))_{i\in \mathbb{N}^+}$ solves \eqref{e202} with the control $u_d$ and the initial datum $y(0)$.
Next, we define the following quadratic cost functional:
\begin{equation*}\label{e203}
J(u_d;y_0):=\sum_{i=1}^{+\infty}(\|y_i(y_0, u_d)\|_Y^2+\|u_i\|_U^2).
\end{equation*}
  Finally, we consider the discrete LQ problem (\textbf{d-LQ}). For any $y_0\in Y$, find $u^*_d=(u_i^*)_{i\in\mathbb{N}^+}\in l^2(\mathbb{N}^+;U)$ such that
\begin{equation*}\label{yu-6-1-1}
    J(u^*_d;y_0)=\inf_{u_d\in l^2(\mathbb{N^+};U)}J(u_d;y_0):=V(y_0).
\end{equation*}
{\it We call the problem (\textbf{d-LQ}) solvable if $V(y_0)<+\infty$ for any $y_0\in Y$.}

The main result of this section is as follows.
\begin{proposition}\label{lm202}
Let  $N\in\mathbb{N}^+$ and $\delta\in(0,1)$. Assume that there is a $C(T, N, \delta)>0$ such that for each $y_0\in Y$, there is an $u^{y_0}\in L_T^2(0,NT; U)$ such that
 \begin{equation}\label{e201}
   \frac{1}{C(T, N, \delta)}\|u^{y_0}\|_{L^2(0,NT; U)}^2+\frac{1}{\delta}\|y(NT; y_0, u^{y_0})\|_Y^2\leq \|y_0\|_Y^2,
\end{equation}
    where $y(\cdot; y_0, u^{y_0})$ is the solution of (\ref{yu-5-26-1}) corresponding to the control $u^{y_0}$ and the initial condition $y(0)=y_0$. Then, the problem (\textbf{d-LQ}) is solvable.
\end{proposition}
   To prove Proposition \ref{lm202}, we need the following lemma.
\begin{lemma}\label{yu-lemma-6-1-1}
    The problem (\textbf{d-LQ}) is solvable if and only if
\begin{equation*}\label{yu-6-1-2}
\sup_{n\in\mathbb{N}^+}\inf_{u_d\in l^2(\mathbb{N}^+;U)}J_n( u_d; y_0)<+\infty\;\;\mbox{for each}\;\;y_0\in Y,
\end{equation*}
    where
\begin{equation}\label{yu-6-1-3}
    J_n( u_d; y_0):=\sum_{i=1}^{n}(\|y_i(y_0, u_d)\|_Y^2+\|u_i\|_U^2),
    \;\;\;u_d=(u_i)_{i\in\mathbb{N}^+}\in l^2(\mathbb{N}^+;U),
\end{equation}
where $(y_i(y_0, u_d))_{i\in\mathbb{N}^+}$ solves the system \eqref{e202} with the control $u_d$ and the initial datum $y_0$.
\end{lemma}
\begin{proof}
We begin by proving the ``only if'' part.
  We assume that the problem (\textbf{d-LQ}) is solvable, i.e., $V(y_0)<+\infty$ for all $y_0\in Y$.
  We arbitrarily fix $y_0\in Y$. Then, for each $n\in\mathbb{N}^+$,
$$
    J_n(u_d; y_0)\leq J(u_d; y_0)\;\;\mbox{for any}\;\;u_d\in l^2(\mathbb{N}^+;U),
$$
which yields
$$
    \inf_{u_d\in l^2(\mathbb{N}^+;U)}J_n(u_d; y_0)\leq \inf_{u_d\in l^2(\mathbb{N}^+;U)}J(u_d; y_0)
    \;\;\mbox{for each}\;\;n\in\mathbb{N}^+.
$$
    Thus, we have
$$
    \sup_{n\in\mathbb{N}^+}\inf_{u_d\in l^2(\mathbb{N}^+;U)}J_n(u_d; y_0)\leq \inf_{u_d\in l^2(\mathbb{N}^+;U)}J(u_d; y_0)<+\infty.
$$

\par
We next prove the ``if'' part.
    We  suppose that $c^*:=\sup_{n\in\mathbb{N}^+}\inf_{u_d\in l^2(\mathbb{N}^+;U)}J_n(u_d; y_0)<+\infty$.
     Then,
\begin{equation}\label{yu-6-1-5}
    \inf_{u_d\in l^2(\mathbb{N}^+;U)}J_{n}(u_d; y_0)\leq  c^*\;\;\mbox{for all}\;\;n\in\mathbb{N}^+
\end{equation}
   We arbitrarily fix an $\varepsilon>0$. It follows from (\ref{yu-6-1-5}) that for each $n\in\mathbb{N}^+$, there exists an $u^\varepsilon_{d,n}=(u^\varepsilon_{nj})_{j\in\mathbb{N}^+}\in l^2(\mathbb{N}^+;U)$ such that
\begin{equation}\label{yu-6-1-6}
    J_n(u_{d,n}^\varepsilon;y_0)\leq c^*+\varepsilon.
\end{equation}
    For each $n\in\mathbb{N}^+$, we let
\begin{equation}\label{yu-6-1-7}
    \hat{u}_{nj}:=
\begin{cases}
    u^\varepsilon_{nj},&\mbox{if}\;\;1\leq j\leq n,\\
    0,&\mbox{if}\;\;j>n,
\end{cases}
    \;\;\mbox{and}\;\;\hat{y}_{nj}
    :=
\begin{cases}
    y_j(y_0, u^\varepsilon_{d,n}),&\mbox{if}\;\;0\leq j\leq n,\\
    0,&\mbox{if}\;\;j>n.
\end{cases}
\end{equation}
    From  \eqref{yu-6-1-7} and  (\ref{yu-6-1-6}), we see that
\begin{equation}\label{yu-6-1-8}
    \|(\hat{y}_{nj})_{j\in\mathbb{N}^+}\|_{l^2(\mathbb{N}^+;Y)}^2
    +\|(\hat{u}_{nj})_{j\in\mathbb{N}^+}\|_{l^2(\mathbb{N}^+;U)}^2\leq c^*+\varepsilon\;\;\mbox{for any}\;\;n\in\mathbb{N}^+.
\end{equation}
Thus, on subsequences denoted in the same manner,
\begin{equation}\label{yu-6-1-9}
    \left((\hat{u}_{nj})_{j\in\mathbb{N}^+},(\hat{y}_{nj})_{j\in\mathbb{N}^+}\right) \to \left((\tilde{u}_j)_{j\in\mathbb{N}^+},(\tilde{y}_j)_{j\in\mathbb{N}^+}\right)\;\;\mbox{weakly in}\;\;l^2(\mathbb{N}^+;U)\times l^2(\mathbb{N}^+;Y)
    \ \mbox{as}\;\;n\to+\infty.
\end{equation}
   From  \eqref{yu-6-1-9}, (\ref{yu-6-1-8}), and the weak lower semi-continuity of $\|\cdot\|_{l^2(\mathbb{N}^+;U)\times l^2(\mathbb{N}^+;Y)}$, we obtain
\begin{equation}\label{yu-6-1-9-b}
    \|(\tilde{y}_j)_{j\in\mathbb{N}^+}\|^2_{l^2(\mathbb{N}^+:Y)}+\|(\tilde{u}_j)_{j\in\mathbb{N}^+}\|^2_{l^2(\mathbb{N}^+;U)}\leq c^*+\varepsilon.
\end{equation}
  Meanwhile, by \eqref{e202},  (\ref{yu-6-1-7}), and (\ref{yu-6-1-9}), we find
  that
     $\tilde{u}_d=(\tilde{u}_j)_{j\in\mathbb{N}^+}$ and $\tilde{y}=(\tilde{y}_j)_{j\in\mathbb{N}^+}$ verify
\begin{equation*}\label{yu-6-1-10}
\begin{cases}
    \tilde{y}_{j}=\Phi \tilde{y}_{j-1}+D\tilde{u}_j,\;\;j\in\mathbb{N}^+,\\
    \tilde{y}_0=y_0.
\end{cases}
\end{equation*}
    This, together with (\ref{yu-6-1-9-b}), yields
$J(\tilde{u}_d; y_0)\leq c^*+\varepsilon$, which leads to
 $\inf_{u_d\in l^2(\mathbb{N}^+;U)}J(y_0, u_d)\leq c^*+\varepsilon
    <+\infty$.
\par
    Thus, we have completed the proof  of Lemma \ref{yu-lemma-6-1-1}.
\end{proof}

  \vskip 5pt

\begin{proof}[Proof of Proposition \ref{lm202}]
 By Lemma \ref{yu-lemma-6-1-1}, we only need to find an $M>0$ such that
\begin{equation}\label{e204}
\inf_{u_d\in l^2(\mathbb{N}^+;U)}J_n(u_d; y_0)\leq M\|y_0\|^2_Y\;\;\mbox{for all}\;\;y_0\in Y\;\;\mbox{and}\;\; n\in\mathbb{N}^+,
\end{equation}
 where $J_n(\cdot;\cdot)$ is defined by (\ref{yu-6-1-3}).
  Notice that when  $n\in\mathbb{N}^+$ and $y_0\in Y$,
\begin{equation*}\label{yu-bb-6-1-2}
    \inf_{u_d\in l^2(\mathbb{N}^+;U)}J_n(u_d; y_0)\leq \inf_{u_d\in l^2(\mathbb{N}^+;U)}J_m(u_d; y_0)
    \;\;\mbox{for all}\;\;m\geq n.
\end{equation*}
    Thus, to prove \eqref{e204}, we only need to find an $M>0$ such that
\begin{equation}\label{e204-b}
\inf_{u_d\in l^2(\mathbb{N}^+;U)}J_{kN}(u_d;y_0)\leq M\|y_0\|^2_Y\;\;\mbox{for all}\;\;y_0\in Y\;\;\mbox{and}\;\; k\in\mathbb{N}^+.
\end{equation}
\par
   We now prove \eqref{e204-b}. We arbitrarily fix a $k\in\mathbb{N}^+$ and then let $n=kN$.
    According to the assumption of this proposition, for each $y_0\in Y$, there is a  $u^{y_0}\in L_T^2(0,NT; U)$ such that  (\ref{e201}) holds, i.e.,
\begin{equation}\label{yu-6-2-1}
    \|y(NT;y_0,u^{y_0})\|_Y^2\leq \delta\|y_0\|^2_Y\;\;\mbox{and}\;\;
    \|u^{y_0}\|_{L^2(0,NT;U)}\leq C(T,N,\delta)\|y_0\|^2_Y.
\end{equation}
   By making use of the assumption of this proposition again, we see that
    for $y_{N}:=y(NT;y_0,u^{y_0})$,  there is a
    $u^{y_{N}}\in L_T^2(0,NT; U)$ such that
\begin{equation}\label{yu-6-2-2}
    \|y(NT;y_{N},u^{y_{N}})\|_Y^2\leq \delta\|y_{N}\|^2_Y\;\;\mbox{and}\;\;
    \|u^{y_{N}}\|_{L^2(0,NT;U)}\leq C(T,N,\delta)\|y_{N}\|^2_Y.
\end{equation}
By \eqref{yu-6-2-1} and \eqref{yu-6-2-2}, where we let $y_0:=y_{0N}$ and $y_N:=y_{1N}$, and using mathematical induction, we can easily prove with the aid of the assumption of this proposition that
there is a sequence $\{u^{y_{jN}}\}_{j=0}^{k-1}\subset L^2_T(0,NT;U)$ with
$y_{jN}:=y(NT; y_{(j-1)N}, u^{y_{(j-1)N}})$,
 such that for each
$j\in\{0,1,\dots,k-1\}$,
\begin{equation}\label{yu-6-2-3}
    \|y(NT;y_{jN},u^{y_{jN}})\|_Y^2\leq \delta\|y_{jN}\|^2_Y\;\;\mbox{and}\;\;
    \|u^{y_{jN}}\|_{L^2(0,NT;U)}\leq C(T,N,\delta)\|y_{jN}\|^2_Y.
\end{equation}
   Then, we define
\begin{equation}\label{yu-6-2-4}
    u(t):=
\begin{cases}
    u^{y_{0N}}(t),&\mbox{if}\;\;t\in[0,NT),\\
       \cdots\cdots,\\
    u^{y_{(k-1)N}}(t-(k-1)NT),&\mbox{if}\;\;t\in[(k-1)NT,kNT),\\
    0,&\mbox{if}\;\;t\in[kNT,+\infty).
\end{cases}
\end{equation}
 It follows from  (\ref{yu-6-2-3}) and (\ref{yu-6-2-4}) that $u\in L_T^2(\mathbb{R}^+;U)$. Hence, we infer by \eqref{yu-6-21-b-1} that for some $ (\hat{u}_i)_{i\in \mathbb{N}^+}\in l^2(\mathbb{N}^+;U)$,
     \begin{equation}\label{2.19wang7-12}
     u(t)=\sum_{i=1}^{+\infty}\chi_{[(i-1)T,iT)}(t)\hat{u}_i,\;\; t\in \mathbb{R}^+.
     \end{equation}
        We write $\hat{y}_i:=y(iT;y_0,u)$ with $i\in \mathbb{N}^+$.
     Then, it follows from (\ref{yu-5-26-1})  and (\ref{2.19wang7-12})
     that
        \begin{equation}\label{yu-6-6-2-b}
    \hat{y}_i=\Phi \hat{y}_{i-1}+D\hat{u}_i\;\;\mbox{for each}\;\;i\in\mathbb{N}^+.
\end{equation}
  Meanwhile, one can see from (\ref{yu-6-2-4}) and (\ref{2.19wang7-12}) that for each $j\in\{1, 2, \dots, k\}$,
\begin{equation}\label{yu-6-2-5}
   y_{jN} =y(jNT;y_0,u)=\hat{y}_{jN},  \;\;\; u^{y_{(j-1)N}}(t)=\sum_{i=1}^{N}\chi_{[(i-1)T, iT)}(t)\hat{u}_{i+(j-1)N}, \;\;t\in [0,NT).
\end{equation}
Now, we can obtain by \eqref{yu-6-2-5} and (\ref{yu-6-2-3}) that
\begin{eqnarray}\label{yu-6-2-6}
    \sum_{i=1}^{kN}\|\hat{u}_i\|^2_U&=&T^{-1}\sum_{j=0}^{k-1}\|u^{y_{jN}}\|_{L^2(0,NT;U)}^2
    \leq T^{-1}C(T,N,\delta)\sum_{j=0}^{k-1}\|y_{jN}\|_Y^2\nonumber\\
    &\leq&(T(1-\delta))^{-1}C(T,N,\delta)\|y_0\|_Y^2.
\end{eqnarray}
Moreover, it follows from \eqref{yu-6-6-2-b} and (\ref{yu-6-2-5})
    that when  $j\in\{0,1,\dots,k-1\}$ and  $jN< i\leq (j+1)N$,
\begin{eqnarray*}\label{yu-6-5-1}
    \|\hat{y}_{i}\|^2_{Y}&\leq&(2\|\Phi\|^2_{\mathcal{L}(Y)})^{i-jN}\|\hat{y}_{jN}\|_Y^2
    +2\|D\|_{\mathcal{L}(U;Y)}^2\sum_{k=1}^{i-jN}(2\|\Phi\|_{\mathcal{L}(Y)}^2)^{i-jN-k}
\|\hat{u}_{jN+k}\|_U^2\nonumber\\
    &\leq&(2+2\|\Phi\|_{\mathcal{L}(Y)}^2)^{i-jN}\|y_{jN}\|_Y^2
    +2 T^{-1}\|D\|_{\mathcal{L}(U;Y)}^2(2+2\|\Phi\|_{\mathcal{L}(Y)}^2)^{i-jN-1}\|u^{y_{jN}}\|_{L^2(0,NT;U)}^2,
\end{eqnarray*}
    which yields
$\sum_{i=jN+1}^{(j+1)N}\|\hat{y}_i\|_Y^2
    \leq C_1
    \|y_{jN}\|_Y^2+C_2\|u^{y_{jN}}\|_{L^2(0,NT;U)}^2$,
    where
    $C_1:=\frac{(2+2\|\Phi\|^2_{\mathcal{L}(Y)})^{N}-1}
    {1+2\|\Phi\|^2_{\mathcal{L}(Y)}}$ and $C_2:=2T^{-1}\|D\|_{\mathcal{L}(U;Y)}^2
    \left(\frac{(2+2\|\Phi\|^2_{\mathcal{L}(Y)})^{N}-1}
    {1+2\|\Phi\|_{\mathcal{L}(Y)}^2}\right)$.
This, along with (\ref{yu-6-2-3}), implies that
\begin{equation*}\label{yu-6-6-1}
    \sum_{i=1}^{kN}\|\hat{y}_i\|_Y^2\leq (C_1+C_2C(T,N,\delta))(1-\delta)^{-1}\|y_0\|_Y^2.
\end{equation*}
    The above, together with (\ref{yu-6-6-2-b}) and (\ref{yu-6-2-6}), leads to
    \begin{equation*}
    \inf_{u_d\in l^2(\mathbb{N}^+;U)}J_{kN}(u_d;y_0)\leq J_{kN}(\hat{u}_d; y_0)\leq M
    \|y_0\|_Y^2,
    \end{equation*}
    where $M:=\left(TC(T,N,\delta)+C_1+C_2C(T,N,\delta)\right)(1-\delta)^{-1}$,
    i.e., (\ref{e204-b}) is true.

   Hence, we have completed the proof of Proposition \ref{lm202}.
\end{proof}

\section{Proofs of main theorems}
    This section proves Theorems \ref{th03}, \ref{th03-1}, and \ref{thm1.11-7-26}.
\subsection{Proof of Theorem \ref{th03}}
      We recall (\ref{yu-6-21-b-1}) for the definitions of the spaces  $L^2_T(0,NT;U)$ ($N\in\mathbb{N}^+$) and $L^2_T(\mathbb{R}^+;U)$.
     The following lemma plays an important role in the proof of Theorem \ref{th03}.
\begin{lemma}\label{lm201}
 Suppose that  there are constants $T>0$,  $N\in\mathbb{N}^+$,  $C>0$, and  $\delta\in(0,1)$ such that the observability inequality (\ref{e107}) holds.
 Then, for each $y_0\in Y$, there is a control $u^{y_0}(\cdot)\in L^2_T(0,NT; U)$ such that
 the inequality (\ref{e201}), with the above $\delta$ and  $C(T, N, \delta)=T^{-1}C$, is true.
\end{lemma}
\begin{proof}
We will use \cite[Lemma 5.1]{Wang-19} to prove this lemma. For this purpose, we define operators $\mathcal{R}: Y\rightarrow Y$ and $\mathcal{O}: Y\rightarrow L^2_T(0,NT; U)$ by
\begin{equation*}
    \mathcal{R} z:=S(NT)^*z;\;\; \mathcal{O}z:=\sum_{i=1}^{N}\chi_{[(i-1)T, iT)}(t)\int_{(i-1)T}^{iT}B^*S(NT-s)^*zds,\;\;z\in Y.
\end{equation*}
Thus, the inequality (\ref{e107}) is the same as
\begin{equation*}
\|\mathcal{R} z\|_Y^2\leq T^{-1}C\|\mathcal{O}z\|_{L^2_T(0,NT;U)}^2+\delta \|z\|_Y^2.
\end{equation*}
Then, according to \cite[Lemma 5.1]{Wang-19}, for each $y_0\in Y$, there is an $u^{y_0}(\cdot)\in L^2_T(0,NT;U)$ such that
\begin{equation*}
 \frac{1}{T^{-1}C}\|-u^{y_0}\|_{L^2_T(0,NT;U)}^2+\frac{1}{\delta}\|\mathcal{R}^*y_0
-\mathcal{O}^*(-u^{y_0})\|_{Y}^2\leq  \|y_0\|_Y^2.
\end{equation*}
Meanwhile, by a simple calculation, we obtain $\mathcal{R}^*y_0-\mathcal{O}^*(-u^{y_0})=y(NT; y_0, u^{y_0})$. This, together with the above inequality, leads to (\ref{e201}), with the above $\delta$ and  $C(T, N, \delta)=T^{-1}C$. Hence, we have completed the proof of Lemma \ref{lm201}.
\end{proof}
We are now in a position to prove Theorem \ref{th03}.
\begin{proof}[Proof of Theorem \ref{th03}]
We arbitrarily fix $T>0$ and  organize the proof in two steps.
\vskip 5pt
  \noindent  \emph{Step 1. We prove  $(i)\Rightarrow (ii)$.}

    Suppose that $(i)$ holds, i.e., there is an $F\in \mathcal{L}(Y;U)$, an $\omega>0$,
    and a $C(\omega)\geq 1$ such that for an arbitrarily fixed  $y_0\in Y$, the solution $y_F(\cdot)$ to the
    system  \eqref{yu-6-18-2} with the initial condition $y(0)=y_0$,
satisfies
\begin{equation}\label{yu-6-8-2}
    \|y_F(t)\|_Y\leq C(\omega)e^{-\omega t}\|y_0\|_Y,\;\;\;t\in\mathbb{R}^+.
\end{equation}
   Let
\begin{equation}\label{yu-6-8-3}
    u_F(t):=F(H_Ty_F)(t)=\sum_{i=1}^{\infty}\chi_{[(i-1)T, iT)}(t)Fy_F((i-1)T),\;\;\;t\in\mathbb{R}^+,
\end{equation}
    where $H_T$ is given by  (\ref{yu-6-25-2}). We arbitrarily fix $N\in \mathbb{N}^+$. It follows from (\ref{yu-6-8-2})
    and \eqref{yu-6-8-3} that $u_F\in L^2_T(\mathbb{R}^+;U)$ and
\begin{equation*}\label{yu-6-8-4}
    y_F(NT)=S(NT)y_0+\int_0^{NT}S(NT-t)Bu_F(t)dt,
\end{equation*}
which leads to
  \begin{eqnarray}\label{e225}
\langle  y_F(NT), \varphi\rangle_Y=\langle  y_0, S(NT)^*\varphi\rangle_Y+\int_0^{NT}\langle  u_F(t), B^*S(NT-t)^*\varphi\rangle_U dt\;\;\mbox{for any}\;\;
\varphi\in Y.
 \end{eqnarray}
   Meanwhile, it follows from (\ref{yu-6-8-3}) and (\ref{yu-6-8-2}) that, for any $\varphi\in Y$,
  \begin{eqnarray*}\label{e226}
&\;&\Big|\int_0^{NT}\langle u_F(t), B^*S(NT-t)^*\varphi\rangle_U dt\Big|\nonumber\\
&=&\Big|\sum_{i=1}^{N}\Big\langle Fy_F((i-1)T), \int_{(i-1)T}^{iT}B^*S(NT-t)^*\varphi dt\Big\rangle_U\Big|\nonumber\\
&\leq&C(\omega)\|F\|_{\mathcal{L}(Y;U)}\Big(\sum_{i=0}^{N-1}e^{-2\omega iT}\Big)^{1/2}\Big(\sum_{i=1}^{N}\Big\|\int_{(i-1)T}^{iT}B^*S(NT-t)^*\varphi dt\Big\|_U^2\Big)^{1/2}\|y_0\|_Y\nonumber\\
&\leq& C(F,\omega)\Big(\sum_{i=1}^{N}\Big\|\int_{(i-1)T}^{iT}B^*S(NT-t)^*\varphi dt\Big\|_U^2\Big)^{1/2}\|y_0\|_Y,
 \end{eqnarray*}
    where $C(F,\omega):=C(\omega)(1-e^{-2\omega T})^{-\frac{1}{2}}\|F\|_{\mathcal{L}(Y;U)}$.
This, along with (\ref{e225}) and (\ref{yu-6-8-2}), yields that for any $\varphi\in Y$,
\begin{eqnarray*}\label{e227}
&\;&\Big|\langle  y_0, S(NT)^*\varphi\rangle_Y\Big|\nonumber\\
&\leq& C(F,\omega)\Big(\sum_{i=1}^{N}\Big\|\int_{(i-1)T}^{iT}B^*S(NT-t)^*\varphi dt\Big\|_U^2\Big)^{1/2}\|y_0\|_Y+C(\omega)e^{-\omega NT} \|y_0\|_Y\|\varphi\|_Y.
 \end{eqnarray*}
   Since $y_0\in Y$ and $N\in\mathbb{N}^+$ were arbitrarily selected, the above  leads to  (\ref{e107-b}), with $C_1=2(C(F,\omega))^2$ and $C_2:=2(C(\omega))^2$. Hence, $(ii)$ is true.

\vskip 5pt
  \noindent  \emph{Step 2. We prove $(ii)\Rightarrow (iii)$.}

   Suppose that $(ii)$ holds, i.e., there are constants  $\omega>0$,  $C_1>0$, and $C_2>0$ such that (\ref{e107-b}) holds for all $N\in\mathbb{N}^+$. Let $N^*\in\mathbb{N}^+$ be such that $C_2e^{-\omega N^*T}<1$. Then, (\ref{e107-b}) with $N=N^*$ leads to
    (\ref{e107}) with $N:=N^*$, $C:=C_1$, and $\delta:= C_2e^{-\omega N^*T}$.
   Hence,
    $(iii)$ is true.
\vskip 5pt
  \noindent   \emph{Step 3. We prove $(iii)\Rightarrow (i)$.}

   Suppose that $(iii)$ holds, i.e., there are constants $N\in \mathbb{N}^+$, $\delta\in (0,1)$, and $C>0$ such that the inequality (\ref{e107}) holds. Then, it follows from Lemmas \ref{lm201} and \ref{yu-lemma-6-1-1} and Proposition \ref{lm202} that
\begin{equation*}\label{e218}
\sup_{n\in\mathbb{N}^+}\inf_{u\in l^2(\mathbb{N}^+;U)}J_n(u;y_0)<\infty\;\;\mbox{for any}\;\;y_0\in Y,
\end{equation*}
where $J_n(\cdot;\cdot)$ is given by (\ref{yu-6-1-3}). This, together with \cite[Theorem 6.1]{Zabczyk-74}, implies that there exists a symmetric operator $K\in \mathcal{L}(Y)$ satisfying $\langle Ky,y\rangle_Y\geq 0$ for any $y\in Y$ and \begin{equation*}\label{e219}
K=\Phi^*K(I+DD^*K)^{-1}\Phi+I,
\end{equation*}
    where $\Phi$ and $D$ are given by \eqref{2.2wang7-11}.
             Moreover,  from  \cite[Theorem 6.2]{Zabczyk-74}, as well as its proof, we obtain that the operator
\begin{equation}\label{e22-1-1}
    F_K:=-(I+D^*K D)^{-1}D^*K\Phi\in \mathcal{L}(Y; U)
\end{equation}
     satisfies
\begin{equation}\label{e220}
\verb"r"(\Phi+DF_K)<1,
\end{equation}
where $\verb"r"(\Phi+DF_K)$ is the spectral radius of $\Phi+DF_K$. Define $r_0:=(1+\verb"r"(\Phi+DF_K))/2$.
Then, it follows from (\ref{e220}) that
\begin{equation}\label{yu-6-6-6}
\verb"r"(\Phi+DF_K)<r_0<1.
\end{equation}
\par
    Next, we arbitrarily fix $y_0\in Y$ and then consider the system (\ref{e202}) with $u_i:=F_Ky_{i-1}$ and the initial datum $y_0$. It is clear that
        \begin{equation}\label{yu-6-6-7}
    y_i=(\Phi+DF_K)y_{i-1}\;\;\mbox{for each}\;\;i\in\mathbb{N}^+.
\end{equation}
   Since  $\verb"r"(\Phi+DF_K))=\lim_{n\to+\infty}\|(\Phi+DF_K)^n\|_{\mathcal{L}(Y)}^{\frac{1}{n}}$, it follows
   from  \eqref{yu-6-6-7} and (\ref{yu-6-6-6})  that there is an $n^*:=n^*(r_0)\in\mathbb{N}^+$ such that
\begin{equation}\label{yu-6-6-8}
    \|y_i\|_Y=\|(\Phi+DF_K)^iy_0\|_Y\leq r_0^{i}\|y_0\|_{Y}\;\;\mbox{for any}\;\;i\geq n^*.
\end{equation}
   From \eqref{yu-6-6-8}, we see that $(y_i)_{i\in\mathbb{N}^+}\in l^2(\mathbb{N}^+;Y)$,
   which, together with the facts that
       $F_K\in \mathcal{L}(Y;U)$
       and $u_i=F_Ky_{i-1}$, yields
           $(u_i)_{i\in\mathbb{N}^+}\in l^2(\mathbb{N}^+;U)$.
\par
   We are now in a position to show  $(i)$, i.e., the system (\ref{yu-5-26-1}) is \textbf{(DC)}$_T$-stabilizable.
    It suffices to prove
     that any solution
     $y_{F_K}(\cdot)$ to the closed-loop system
\begin{equation}\label{yu-6-7-1}
    y'(t)=Ay(t)+BF_K(H_Ty)(t),\;\; t\in\mathbb{R}^+
\end{equation}
 is exponentially stable, where $F_K$ is given by (\ref{e22-1-1}).
  Indeed, by \eqref{yu-6-7-1} and
   (\ref{yu-6-25-2}), one can easily check that
\begin{equation*}\label{yu-6-7-2}
    y_{F_K}(iT)=\Big(S(T)+\int_0^TS(T-s)dsBF_K\Big) y_{F_K}((i-1)T)
    \;\;\mbox{for each}\;\;i\in\mathbb{N}^+.
\end{equation*}
    This, along with \eqref{2.2wang7-11}, (\ref{yu-6-6-7}), and (\ref{yu-6-6-8}), leads to
\begin{equation}\label{yu-6-7-3}
    \|y_{F_K}(iT)\|_Y\leq r_0^{i}\|y_{F_K}(0)\|_Y\;\;\mbox{for each}\;\;i\geq n^*.
\end{equation}
     Meanwhile, for an arbitrarily fixed $m^*\in \mathbb{N}$, we can directly check  that when $t\in[m^*T,(m^*+1)T)$, \begin{eqnarray}\label{yu-6-7-4}
    \|y_{F_K}(t)\|_Y=\Big\|\Big(S(t-m^*T)+\int_{m^*T}^tS(t-s)BF_Kds\Big) y_{F_K}(m^*T)\Big\|_Y
    \leq C(K)\|y_{F_K}(m^*T)\|_Y,
\end{eqnarray}
    where
    $C(K):=(1+T\|B\|_{\mathcal{L}(U;Y)}\|F_K\|_{\mathcal{L}(Y;U)})\sup_{t\in [0,T]}\|S(t)\|_{\mathcal{L}(Y)}$.
    By \eqref{yu-6-7-4} and  (\ref{yu-6-7-3}), we obtain  that for any $t\in[m^*T,(m^*+1)T)$,
\begin{equation}\label{yu-b-6-25-1}
    \|y_{F_K}(t)\|_Y\leq
\begin{cases}
    (C(K))^{n^*}e^{\omega_Kn^*T}e^{-\omega_Kt}\|y_{F_K}(0)\|_Y,\;\;\mbox{if}\;\;0\leq m^*<n^*,\\
    C(K)r_0^{-1}e^{-\omega_Kt}\|y_{F_K}(0)\|_Y,\;\;\mbox{if}\;\;m^*\geq n^*,
\end{cases}
\end{equation}
   where $\omega_K:=-T^{-1}\ln r_0$ is positive because of (\ref{yu-6-6-6}).

   Finally, it follows from \eqref{yu-b-6-25-1} that the system \eqref{yu-6-7-1}
      is exponentially stable. Thus,
      $(i)$ is true.
\par
   Hence, we have completed the proof of Theorem \ref{th03}.
\end{proof}

%\end{proof}
%\subsection{Proof of Theorem \ref{yu-theorem-6-27-1}}
%    Let $T>0$ be fixed. Since $\mathcal{L}(Y;U)\subset \mathcal{PC}_T$, the claim $(i)\Rightarrow (ii)$ is trivial. Thus, we only need to show that
%    $(ii)\Rightarrow (i)$ is true. We prove it by Theorem \ref{th03}. Actually, by Theorem \ref{th03}, if we can prove that there is $N\in\mathbb{N}^+$, $\delta\in(0,1)$, and $C(T,N,\delta)\geq 0$ such that (\ref{e107}) holds, then the proof is completed.
\subsection{Proof of Theorem  \ref{th03-1}}
We organized the proof in several steps.

\vskip 5pt
   \noindent \emph{Step 1. It is trivial that $(i)\Rightarrow (ii)$.}
\vskip 5pt
   \noindent \emph{Step 2. We prove that $(ii)\Rightarrow (iii)$.}

    Suppose that $(ii)$ holds, i.e., there is a $T>0$, an $\omega>0$, a $C(\omega, T)>0$, and an $F(\cdot)\in \mathcal{F}_{T,\infty}$ such that for an arbitrarily fixed $y_0\in Y$, the solution $y_F(\cdot)$ to the
    equation \eqref{yu-6-21-1}, with the initial condition $y(0)=y_0$,
    satisfies
\begin{equation}\label{yu-6-15-9}
    \|y_F(t)\|_Y\leq C(\omega,T)e^{-\omega t}\|y_0\|_Y, \;\;t\in\mathbb{R}^+.
\end{equation}
   Let
\begin{equation}\label{yu-6-15-10}
    u_F(t):=F(t)(H_Ty_F)(t)\;\;\mbox{a.e.}\;\;t\in\mathbb{R}^+.
\end{equation}
Several facts are stated as follows. First,
    it follows from (\ref{yu-6-25-2}) that
    $u_F(t)=F(t)y_F([t/T]T),\; t\in \mathbb{R}^+$. Second, since $F(\cdot)\in \mathcal{F}_{T,\infty}$, it follows from (\ref{yu-6-15-9}) that
\begin{eqnarray}\label{yu-6-15-11}
    \int_0^{+\infty}\|u_F(t)\|_U^2dt&=&\sum_{i=1}^{+\infty}
    \int_{(i-1)T}^{iT}\|F(t)y_F((i-1)T)\|^2_Udt\nonumber\\
    &\leq&(C(\omega,T))^2\int_{0}^{T}
    \|F(t)\|^2_{\mathcal{L}(Y;U)}dt\sum_{i=1}^{+\infty}e^{-2(i-1)\omega T}\|y_0\|_Y^2\nonumber\\
    &=& \widetilde{C}(\omega,T,F)\|y_0\|_Y^2,
\end{eqnarray}
   where $\widetilde{C}(\omega,T,F):=T(1-e^{-2\omega T})^{-1}(C(\omega, T))^{2} \|F(\cdot)\|^2_{L^\infty(0, T; \mathcal{L}(Y;U))}$. Thus,  $u_F\in L^2(\mathbb{R}^+;U)$.
   Third, it follows from \eqref{yu-6-21-1} and (\ref{yu-6-15-10}) that
\begin{equation}\label{yu-6-15-12}
    y_F(t)=S(t)y_0+\int_0^tS(t-s)Bu_{F}(s)ds\;\;\mbox{for any}\;\;t\in\mathbb{R}^+.
\end{equation}

    Now, by (\ref{yu-6-15-9}), (\ref{yu-6-15-11}), and
     (\ref{yu-6-15-12}), we have that when $\hat{N}\in\mathbb{N}^+$  and  $\varphi\in Y$,
\begin{eqnarray}\label{yu-6-15-137-12}
    &\;&|\langle y_0, S(\hat{N}T)^*\varphi\rangle_Y|\nonumber\\
    &\leq & |\langle y_F(\hat{N}T), \varphi\rangle_Y|+\Big|\int_0^{\hat{N}T}\langle u_{F}(s), B^*S(\hat{N}T-s)^*\varphi\rangle_Uds\Big|\nonumber\\
 &\leq& \|y_F(\hat{N}T)\|_Y\|\varphi\|_Y+\Big(\int_0^{\hat{N}T}\|u_{F}(s)\|_U^2ds\Big)
 ^{\frac{1}{2}}\Big(\int_0^{\hat{N}T}\|B^*S(\hat{N}T-s)^*
 \varphi\|_U^2ds\Big)^{\frac{1}{2}}\nonumber\\
 &\leq& C(\omega, T)e^{-\hat{N}\omega T} \|y_0\|_Y\|\varphi\|_Y+\sqrt{\widetilde{C}(\omega, T,F)}\|y_0\|_Y\Big(\int_0^{\hat{N}T}\|B^*S(\hat{N}T-s)^*\varphi\|_U^2ds\Big)^{\frac{1}{2}}.
\end{eqnarray}
Noting that $y_0\in Y$ was arbitrarily taken, by taking $\hat{N}\in\mathbb{N}^+$ such that
\begin{equation*}\label{yu-6-15-bb-1}
 \hat{\delta}:=2(C(\omega, T))^2e^{-2\hat{N}\omega T}\in (0,1)
\end{equation*}
in \eqref{yu-6-15-137-12},
we obtain
   \begin{equation*}\label{yu-6-15-14}
    \|S(\hat{N}T)^*\varphi\|^2_Y\leq 2\widetilde{C}(\omega, T,F)\int_0^{\hat{N}T}\|B^*S(\hat{N}T-s)^*\varphi\|_U^2ds+ \hat{\delta} \|\varphi\|_Y^2\;\;\mbox{for any}\;\;\varphi\in Y.
\end{equation*}
   This implies that the inequality (\ref{yu-6-18-3}) holds with $T = \hat{N}T, C = 2\widetilde{C}(\omega, T,F)$ and $\delta = \hat{\delta}$. Thus,
    we  conclude that $(iii)$ holds.

    \vskip 5pt
\noindent \emph{Step 3. We prove $(iii)\Rightarrow (i)$.}

 Suppose that $(iii)$ holds. We know by \cite[Theorem 1]{Trelat-20} that the system (\ref{yu-5-26-1}) is \textbf{(CC)}-stabilizable, i.e.,  there is an $\omega>0$, a $C(\omega)>0$, and an $F\in \mathcal{L}(Y; U)$  such that  for each  $y_0\in Y$, the  solution $y_F(\cdot;y_0)$ to the equation \eqref{yu-6-18-17-12}
 with the initial condition $y(0)=y_0$
satisfies
\begin{equation}\label{yu-6-14-1}
  \|y_F(t;y_0)\|_Y\leq C(\omega)e^{-\omega t}\|y_0\|_Y\;\;\mbox{for any}\;\;t\in\mathbb{R}^+.
\end{equation}
  We arbitrarily fix $y_0\in Y$.  Several facts are stated as follows. First, because $BF\in\mathcal{L}(Y)$, the operator $A+BF$ with
     the domain $D(A+BF)=D(A)$ generates a $C_0$-semigroup $S_F(\cdot)$ on $Y$
      (\cite[Theorem 1.1, Chapter 3]{Pazy}). This, together with (\ref{yu-6-14-1}), yields
\begin{equation}\label{yu-6-15-1}
    \|S_F(t)\|_{\mathcal{L}(Y)}\leq C(\omega)e^{-\omega t},\;\;\;t\in\mathbb{R}^+.
\end{equation}
   Second, it is clear that
\begin{equation}\label{yu-6-15-2}
    y_F(t;y_0)=S_F(t)y_0=S(t)y_0+\int_0^tS(t-s)BFy_F(s;y_0)ds,\;\;\;t\in\mathbb{R}^+.
\end{equation}
  Third, for an arbitrarily fixed  $T>0$, we define $F_T(\cdot): \mathbb{R}^+\rightarrow \mathcal{L}(Y; U)$ by
\begin{equation}\label{e232}
F_T(t):=FS_F(t-[t/T]T),\;\; t\in\mathbb{R}^+.
\end{equation}
   Then, we have $F_T(\cdot)\in \mathcal{F}_{T,\infty}$, because  $F\in\mathcal{L}(Y;U)$.

    We claim the exponential stability of the following system:
\begin{equation}\label{yu-6-14-2}
    z'(t)=Az(t)+BF_T(t)(H_Tz)(t),\;\; t\in\mathbb{R}^+,
\end{equation}
  where $H_T$ is given by (\ref{yu-6-25-2}).

  When the above claim is proven, we can use it, as well as the facts that $T$
  was arbitrarily taken and $F_T(\cdot)\in \mathcal{F}_{T,\infty}$, to
  conclude that the system (\ref{yu-5-26-1}) is \textbf{(DP)}$_{T}$-stabilizable for any $T>0$, i.e., $(ii)$ is true.

   The remaining task is to show the above claim. First,
    we   write  $z_{F_T}(\cdot;y_0)\in C([0,+\infty);Y)$ for the solution to
  the equation \eqref{yu-6-14-2} with the initial condition $z(0)=y_0$.

 Next, we are in a position to prove that for all $k\in\mathbb{N}$,
\begin{equation}\label{yu-6-14-3}
    z_{F_T}(kT;y_0)=y_F(kT;y_0).
\end{equation}
   This will be proven by mathematical induction. Since \eqref{yu-6-14-3} with $k=0$ is true, we only need to show \eqref{yu-6-14-3} with $k\in \mathbb{N}^+$.
   To prove \eqref{yu-6-14-3} with $k=1$, we use  (\ref{yu-6-25-2}), (\ref{yu-6-14-2}), (\ref{e232}), and (\ref{yu-6-15-2}) to show that when $t\in [0,T)$,
\begin{eqnarray}\label{yu-6-14-4}
    z_{F_T}(t; y_0)&=&S(t)y_0+\int_{0}^{t}S(t-s)BF_T(s)y_0ds\nonumber\\
&=&S(t)y_0+\int_{0}^{t}S(t-s)BFS_F(s)y_0ds\nonumber\\
&=&S(t)y_0+\int_{0}^{t}S(t-s)BFy_F(s;y_0)ds
= y_F(t; y_0).
\end{eqnarray}
         Because  $z_{F_T}(\cdot;y_0)$ and  $y_F(\cdot; y_0)$ are continuous over $[0,+\infty)$, we
         obtain from \eqref{yu-6-14-4} that  $z_{F_T}(T;y_0)=y_F(T;y_0)$, which leads to (\ref{yu-6-14-3})  with $k=1$.
                  We suppose by induction that $z_{F_T}(k_0T; y_0)=y_F(k_0T; y_0)$ for some   $k_0\in\mathbb{N}^+$. Then, by a similar argument to that used to prove (\ref{yu-6-14-4}), we obtain
\begin{equation*}\label{yu-6-15-3}
    z_{F_T}(t;z_{F_T}(k_0T;y_0))=y_F(t;y_F(k_0T;y_0)),\;\;\;t\in[0, T).
\end{equation*}
    Since $z_{F_T}(\cdot;y_0)$ and $y_F(\cdot; y_0)$ are continuous over $\mathbb{R}^+$, the above,
    along with the assumption that $z_{F_T}(k_0T; y_0)=y_F(k_0T; y_0)$, implies that
\begin{equation}\label{yu-6-15-4}
    z_{F_T}(T;z_{F_T}(k_0T;y_0))=y_F(T;y_F(k_0T;y_0)).
\end{equation}
    However, it is clear that
\begin{equation*}\label{yu-6-15-5}
    z_{F_T}(T;z_{F_T}(k_0T;y_0))=z_{F_T}((k_0+1)T;y_0)\;\;\mbox{and}\;\;
    y_F(T;y_F(k_0T;y_0))=y_F((k_0+1)T;y_0).
\end{equation*}
    These, together with (\ref{yu-6-15-4}), yields $z_{F_T}((k_0+1)T;y_0)=y_F((k_0+1)T;y_0)$,
    i.e., (\ref{yu-6-14-3}) with $k=k_0+1$ is true.
    Thus, by mathematical induction, we conclude that (\ref{yu-6-14-3}) is true for all $k\in\mathbb{N}$.

Now, it follows from  (\ref{yu-6-15-1}), the first equality in (\ref{yu-6-15-2}), and (\ref{yu-6-14-3})
   that
\begin{equation}\label{yu-6-15-6}
    \|z_{F_T}(kT;y_0)\|_Y\leq C(\omega)e^{-k\omega T}\|y_0\|_Y\;\;\mbox{for each}\;\;k\in\mathbb{N}.
\end{equation}
  We arbitrarily fix $t\in\mathbb{R}^+$. Then, there is a $j^*\in\mathbb{N}$ such that $t\in[j^*T,(j^*+1)T)$. This, along with (\ref{yu-6-15-1}) and (\ref{yu-6-15-6}), yields
\begin{eqnarray*}\label{yu-6-15-7}
    \|z_{F_T}(t;y_0)\|_Y&=&\|z_{F_T}(t-j^*T;z_{F_T}(j^*T;y_0))\|_Y=
    \|S_F(t-j^*T)z_{F_T}(j^*T;y_0)\|_Y\nonumber\\
    &\leq&(C(\omega))^2e^{-j^*\omega T}\|y_0\|_Y\leq \widehat{C}(\omega,T)e^{-\omega t}\|y_0\|_Y,
\end{eqnarray*}
    where $\widehat{C}(\omega,T):=(C(\omega))^2e^{\omega T}$.
     Since $t\in \mathbb{R}^+$ and $y_0\in Y$ were arbitrarily taken, the above leads to the exponential stability of
      the  system (\ref{yu-6-14-2}).

\par
    Thus, we have completed the proof of Theorem  \ref{th03-1}.

\subsection{Proof of Theorem \ref{thm1.11-7-26}}

We organize the proof in several steps.

\vskip 5pt
   \noindent \emph{Step 1. We prove  $(i)$.}

    Arbitrarily fix a  $T>0$, and suppose that the system (\ref{yu-5-26-1}) is \textbf{(DC)}$_T$-stabilizable. Then,
 according to Theorem \ref{th03}, there are constants $N\in \mathbb{N}^+$, $\delta\in (0,1)$, and $C\geq 0$ such that the inequality (\ref{e107}) holds.
 Meanwhile, by  H\"{o}lder's inequality, we find that
\begin{eqnarray}\label{e235}
\sum_{i=1}^{N}\Big\|\int_{(i-1)T}^{iT}B^*S(t)^*\varphi dt\Big\|_U^2
\leq T \sum_{i=1}^{N}\int_{(i-1)T}^{iT}\|B^*S(t)^*\varphi\|_U^2 dt
=T\int_0^{NT}\|B^*S(t)^*\varphi\|_U^2 dt,\;\;\varphi\in Y.
\end{eqnarray}
Then, it follows from  (\ref{e107}) and (\ref{e235}) that
\begin{equation*}\label{e236}
\|S(NT)^*\varphi\|_Y^2\leq TC\int_0^{NT}\|B^*S(t)^*\varphi\|_U^2 dt+\delta \|\varphi\|_Y^2,\;\; \varphi\in Y.
\end{equation*}
Since $\delta\in(0,1)$,
the above inequality, together with \cite[Theorem 1]{Trelat-20},
yields that
 the system (\ref{yu-5-26-1}) is  \textbf{(CC)}-stabilizable.

 To prove the reverse is not true, it suffices to provides an example that is \textbf{(CC)}-stabilizable but not \textbf{(DC)}$_T$-stabilizable for any $T>0$. In fact, Example 3 in Subsection \ref{yu-sec-6-22-1} is
 such an example (see Theorem \ref{th301-b}).

\vskip 5pt
   \noindent \emph{Step 2. $(ii)$ follows from Theorem \ref{th03-1} and \cite[Theorem 1]{Trelat-20}
   immediately.}

\vskip 5pt
   \noindent \emph{Step 3. We prove  $(iii)$.}

   It is clear that the \textbf{(CC)}-stabilizability implies the  \textbf{(CP)}$_T$-stabilizability for all $T>0$. Thus, we only
need to show that the \textbf{(CP)}$_T$-stabilizability for some $T>0$ implies the \textbf{(CC)}-stabilizability.
To this end, we  suppose that (\ref{yu-5-26-1}) is \textbf{(CP)}$_T$-stabilizable for some $T>0$,
 i.e., there is an $F(\cdot)\in \mathcal{F}_{T,\infty}$ such that for any $y_0\in Y$, the solution $y_F(\cdot,y_0)$ to the system (\ref{yu-6-28-1wang7-13}) with
     the initial condition $y(0)=y_0$ satisfies
\begin{equation}\label{yu-6-28-1}
    \|y_F(t;y_0)\|_Y\leq C(\omega)e^{-\omega t}\|y_0\|_Y,\;\;\;t\in\mathbb{R}^+,
\end{equation}
    where $\omega>0$ and $C(\omega)>0$. Meanwhile, it is clear that
\begin{equation*}\label{yu-6-28-3}
    y_F(t;y_0)=S(t)y_0+\int_0^tS(t-s)BF(s)y_F(s)ds,\;\;\;t\in\mathbb{R}^+,
\end{equation*}
    which yields that when  $T>0$ and  $\varphi\in Y$,
\begin{equation*}\label{yu-6-28-4}
    \langle y_F(T;y_0), \varphi\rangle_Y=
    \langle y_0, S(T)^*\varphi\rangle_Y+\int_0^T\langle F(s)y_F(s;y_0), B^*S(T-s)^*\varphi
    \rangle_Uds.
\end{equation*}
   This, along with (\ref{yu-6-28-1}), implies that when  $T>0$ and  $\varphi\in Y$,
\begin{eqnarray}\label{yu-6-28-5}
    &\;&|\langle y_0, S(T)^*\varphi\rangle_Y|\nonumber\\
    &\leq&
    \Big|\int_0^T\langle F(s)y_F(s;y_0), B^*S(T-s)^*\varphi
    \rangle_Uds\Big|+|\langle y_F(T;y_0), \varphi\rangle_Y|\nonumber\\
    &\leq& C(\omega, T, F)
    \|y_0\|\|B^*S(T-\cdot)^*\varphi\|_{L^2(0,T; U)}
    +C(\omega)e^{-\omega T}\|y_0\|_Y\|\varphi\|_Y,
\end{eqnarray}
    where $C(\omega, T, F):=C(\omega)\omega^{-1}\|F(\cdot)\|_{L^\infty(0,T;\mathcal{L}(Y;U))}$.
    Let $\widehat{T}>0$ satisfy $\hat{\delta}:=C(\omega)e^{-\omega \widehat{T}}\in (0,1)$.
    Since $y_0\in Y$ was arbitrarily taken, we can use (\ref{yu-6-28-5}) with $T=\widehat{T}$
    to obtain
      (\ref{yu-6-18-3}) with
    $T=\widehat{T}$, $\delta=\hat{\delta}$, and $C=C(\omega, \widehat{T}, F)$.
    This, along with \cite[Theorem 1]{Trelat-20},  leads to the   \textbf{(CC)}-stabilizability.

\vskip 5pt
   \noindent \emph{Step 4. The conclusion  $(iv)$ follows from $(i)$ and $(ii)$ at once.}

\vskip 5pt
   \noindent \emph{Step 5. The conclusion  $(v)$ follows from $(ii)$ and $(iii)$ at once.}

\par
    Thus, we have completed the proof of Theorem  \ref{thm1.11-7-26}.

\section{Applications}\label{yu-sec-6-22}

We present several applications of Theorem \ref{th03}. The first subsection gives some properties on the
\textbf{(DC)}$_T$-stabilizability, while the second subsection shows some examples that are concrete
control differential equations.

 \subsection{Some properties of the \textbf{(DC)}$_T$-stabilizability} \label{liu-7-22}

 \begin{theorem}\label{liu-proposition-7-2-1}
Let
\begin{equation*}\label{liu-7-2-01}
\mathcal{T}:=\{T\in \mathbb{R}^+:\ \mbox{the system}\ (\ref{yu-5-26-1}) \ \mathrm{is}\  \textbf{(DC)}_T-\mathrm{stabilizable}\}.
\end{equation*}
Then, $\mathcal{T}$ is an open subset of  $\mathbb{R}^+$.
\end{theorem}
\begin{proof}
Without loss of generality, we can assume that $\mathcal{T}\neq\emptyset$.
We arbitrarily fix a $T_0\in \mathcal{T}$.
Then, according to Theorem \ref{th03}, there are constants $N\in \mathbb{N}^+$,  $\delta\in (0,1)$, and $C\geq 0$ such that (\ref{e107}) (with $T=T_0$) holds.

We claim that there is an $\eta:=\eta(T_0, N, C, \delta)\in(0,\min\{1,T_0\})$ such that for each $T\in(T_0-\eta, T_0+\eta)$, there are constants  $\tilde{\delta}\in(0, 1)$, $\widetilde{C}\geq 0$, and $n\in \mathbb{N}^+$  such that
\begin{equation}\label{liu-7-2-02}
\|S(nNT)^*\varphi\|_Y^2\leq \widetilde{C}\sum_{i=1}^{nN}\Big\|\int_{(i-1)T}^{iT}B^*S(t)^*\varphi dt\Big\|_U^2+\tilde{\delta} \|\varphi\|^2_Y,\;\;\varphi\in Y.
\end{equation}
When this claim is proven, we can apply Theorem \ref{th03} to conclude that the system (\ref{yu-5-26-1}) is \textbf{(DC)}$_T$-stabilizable for all $T\in(T_0-\eta, T_0+\eta)$, which means that
 $\mathcal{T}$ is open in $\mathbb{R}^+$.

 The remaining task is to prove the above claim. To this end, we first define the following numbers:
  \begin{equation}\label{liu-7-11-001}
M:=\sup_{t\in [0, NT_0]}\|S(t)^*\|^2_{\mathcal{L}(Y)};\;\; m_0:=\Big[\frac{\ln 4M}{-\ln\delta}\Big]+1;\;\;\delta_0:=2M\delta^{m_0};\;\;\eta:=\min\{1, \eta_1, \eta_2\},
\end{equation}
where
   \begin{equation*}\label{liu-7-11-02}
 \eta_1:=\frac{T_0}{m_0+1},\;\; \eta_2:=\frac{\sqrt{\delta_0}}{\sqrt{12M(C+1)(m_0N)^3}(\|B^*\|_{\mathcal{L}(Y; U)}+1)\sup_{t\in[0, m_0N(T_0+1)]}\|S(t)^*\|_{\mathcal{L}(Y)}}.
\end{equation*}
It is clear that $M\geq 1$ and $\delta_0\in(0,1)$. Next, we organize the rest of the proof in several steps.

\vskip 5pt
\noindent \emph{Step 1. We prove that when $m\in \mathbb{N}^+$,
\begin{equation}\label{liu-7-11-04}
\|S(mNT_0)^*\varphi\|_Y^2\leq C\sum_{i=1}^{mN}\Big\|\int_{(i-1)T_0}^{iT_0}B^*S(t)^*\varphi dt\Big\|_U^2+\delta^m \|\varphi\|^2_Y\;\;\mbox{for any}\;\;\varphi\in Y.
\end{equation}}

First, (\ref{liu-7-11-04}), with  $m=1$, is exactly (\ref{e107}) with $T_0\in \mathcal{T}$.
Next, we inductively assume that (\ref{liu-7-11-04}) holds for  $m=k$.
Then, by (\ref{liu-7-11-04}), where $m=1$ or $m=k$, we obtain
that when $\varphi\in Y$,
\begin{eqnarray*}\label{liu-7-11-05}
&&\|S((k+1)NT_0)^*\varphi\|_Y^2=\|S(kNT_0)^*S(NT_0)^*\varphi\|_Y^2\nonumber\\
&\leq& C\sum_{i=1}^{kN}\Big\|\int_{(i-1)T_0}^{iT_0}B^*S(t)^*S(NT_0)^*\varphi dt\Big\|_U^2+\delta^k \|S(NT_0)^*\varphi\|^2_Y\nonumber\\
&\leq& C\sum_{i=N+1}^{(k+1)N}\Big\|\int_{(i-1)T_0}^{iT_0}B^*S(t)^*\varphi dt\Big\|_U^2+\delta^k \Big(C\sum_{i=1}^{N}\Big\|\int_{(i-1)T_0}^{iT_0}B^*S(t)^*\varphi dt\Big\|_U^2+\delta \|\varphi\|^2_Y \Big)\nonumber\\
&\leq& C\sum_{i=1}^{(k+1)N}\Big\|\int_{(i-1)T_0}^{iT_0}B^*S(t)^*\varphi dt\Big\|_U^2+\delta^{k+1} \|\varphi\|^2_Y,
\end{eqnarray*}
which leads to  (\ref{liu-7-11-04}), with $m=k+1$.
 Finally, according to mathematical induction, we see that (\ref{liu-7-11-04}) holds for all $m\in \mathbb{N}^+$.
\vskip 5pt
\noindent
\emph{Step 2. We prove (\ref{liu-7-2-02}), where $T\in (T_0, T_0+\eta)$.}
\par
\vskip 5pt
  We arbitrarily fix a $T\in  (T_0, T_0+\eta)$. Then, by
  \eqref{liu-7-11-001}, we have
     $0<m_0NT-m_0NT_0<NT_0$, which, along with the first equation in   \eqref{liu-7-11-001},
     yields
     \begin{equation*}\label{liu-7-11-06}
\|S(m_0NT-m_0NT_0)^*\|_{\mathcal{L}(Y)}^2\leq M.
\end{equation*}
The above, along with (\ref{liu-7-11-04}) (where $m=m_0$), indicates that
\begin{eqnarray}\label{liu-7-11-07}
\|S(m_0NT)^*\varphi\|_Y^2&\leq&\|S(m_0NT-m_0NT_0)^*\|_{\mathcal{L}(Y)}^2\|S(m_0NT_0)^*\varphi\|_Y^2\nonumber\\
&\leq& MC\sum_{i=1}^{m_0N}\Big\|\int_{(i-1)T_0}^{iT_0}B^*S(t)^*\varphi dt\Big\|_U^2+M\delta^{m_0} \|\varphi\|^2_Y\;\;\mbox{for any}\;\;\varphi\in Y.
\end{eqnarray}
Meanwhile, direct computations, along with \eqref{liu-7-11-001}, indicate that
when $j\in\{1, 2, \ldots, m_0N\}$,
\begin{eqnarray*}\label{liu-7-11-08}
    &\;&\Big\|\int_{(j-1)T_0}^{jT_0}B^*S(t)^*\varphi dt\Big\|_U^2\nonumber\\
    &\leq& 3\Big(\Big\|\int_{jT}^{jT_0}B^*S(t)^*\varphi dt\Big\|^2_U
    +\Big\|\int_{(j-1)T}^{jT}B^*S(t)^*\varphi dt\Big\|^2_U
    +\Big\|\int_{(j-1)T_0}^{(j-1)T}B^*S(t)^*\varphi dt\Big\|^2_U\Big)
    \nonumber\\
    &\leq& 3\Big\|\int_{(j-1)T}^{jT}B^*S(t)^*\varphi dt\Big\|^2_U+6(m_0N)^2\eta^2\|B^*\|^2_{\mathcal{L}(Y; U)}\Big(\sup_{t\in[0, m_0N(T_0+1)]}\|S(t)^*\|_{\mathcal{L}(Y)}\Big)^2\|\varphi\|_Y^2\nonumber\\
      &\leq&3\Big\|\int_{(j-1)T}^{jT}B^*S(t)^*\varphi dt\Big\|^2_U+\frac{ \delta_0}{2M(C+1)m_0N}\|\varphi\|_Y^2.
\end{eqnarray*}
This, together with (\ref{liu-7-11-07}), yields
\begin{equation}\label{liu-7-11-09wang7-25}
\|S(m_0NT)^*\varphi\|_Y^2 \leq 3 MC\sum_{i=1}^{m_0N}\Big\|\int_{(i-1)T}^{iT}B^*S(t)^*\varphi dt\Big\|_U^2+ \delta_0\|\varphi\|^2_Y\;\;\mbox{for any}\;\;\varphi\in Y.
\end{equation}
From \eqref{liu-7-11-09wang7-25}, we are led to  (\ref{liu-7-2-02}), where  $n=m_0$, $\tilde{\delta}=\delta_0$, and
 $T\in  (T_0, T_0+\eta)$.

\vskip 5pt
\noindent\emph{Step 3. We prove (\ref{liu-7-2-02}) where $T\in (T_0-\eta, T_0)$.}
\vskip 5pt
   We arbitrarily fix $T\in  (T_0-\eta, T_0)$.
   Then, by
  \eqref{liu-7-11-001}, we have
    $0<(m_0+1)NT-m_0NT_0<NT_0$. This, along with (\ref{liu-7-11-04}) (where $m=m_0$) and the first equation in \eqref{liu-7-11-001}, yields that when $\varphi\in Y$,
\begin{eqnarray*}\label{liu-7-11-10}
\|S((m_0+1)NT)^*\varphi\|_Y^2
\leq MC\sum_{i=1}^{m_0N}\Big\|\int_{(i-1)T_0}^{iT_0}B^*S(t)^*\varphi dt\Big\|_U^2+M\delta^{m_0} \|\varphi\|^2_Y.
\end{eqnarray*}
Then, by similar arguments to those used in the proof of
\eqref{liu-7-11-09wang7-25}, we obtain
\begin{equation*}\label{liu-7-11-09}
\|S((m_0+1)NT)^*\varphi\|_Y^2 \leq 3 MC\sum_{i=1}^{m_0N}\Big\|\int_{(i-1)T}^{iT}B^*S(t)^*\varphi dt\Big\|_U^2+ \delta_0 \|\varphi\|^2_Y\;\;\mbox{for any}\;\;\varphi\in Y.
\end{equation*}
This, along with (\ref{liu-7-11-001}), leads to  (\ref{liu-7-2-02}),
where $n=m_0+1$, $\tilde{\delta}:=\delta_0$, and $T\in  (T_0-\eta, T_0)$.

Thus, we have completed the proof of Theorem \ref{liu-proposition-7-2-1}.
\end{proof}

As a direct consequence of Theorem \ref{liu-proposition-7-2-1}, we have the following.

\begin{corollary}\label{liu-corollary-7-2-1}
If $\mathbb{R}^+\setminus \mathcal{T}$ is dense in $\mathbb{R}^+$, then $\mathcal{T}=\emptyset$.
\end{corollary}

\subsection{Examples }\label{yu-sec-6-22-1}
\textbf{Example 1.} \emph{(Harmonic oscillator.)}
   We consider the control harmonic oscillator:
\begin{equation}\label{yu-6-22-1}
    \frac{d}{dt}\left(\begin{array}{c}
      y_1(t) \\
      y_2(t)
    \end{array}\right)=A\left(\begin{array}{c}
      y_1(t) \\
      y_2(t)
    \end{array}\right)+Bu(t),
    \;\;t\in\mathbb{R}^+,
\end{equation}
    where $A:=\left(
                  \begin{array}{cc}
                    0 & 1 \\
                    -1 & 0 \\
                  \end{array}
                \right)$, $B:=\left(\begin{array}{c}
                  0 \\
                  1
                \end{array}\right)$, and $u\in L^2(\mathbb{R}^+;\mathbb{R})$.
   Several facts deserve to be mentioned. First, the system \eqref{yu-6-22-1}
   can be put into our framework (\ref{yu-5-26-1}) by setting  $Y:=\mathbb{R}^2$ and $U:=\mathbb{R}$.
   Second, since $\mbox{Rank}(B, AB)=2$, the system \eqref{yu-6-22-1} is controllable. Consequently, it is  \textbf{(CC)}-stabilizable. Third, since $\sigma(A)=\{\verb"i", -\verb"i"\}$, we have that $T$ is pathological (w.r.t. $A$)\footnote{Given an $n\times n$ matrix $A$, we call  $T>0$ pathological (w.r.t. A) if
    $T=2k\pi/|\mathrm{Im}(\lambda_i-\lambda_j)|$, where
      $\lambda_i$ and $\lambda_j$ are two distinct eigenvalues (of $A$) having the same non-negative real parts, and $k\in\mathbb{N}^+$ (\cite{Chen-96}). }
    if and only if $T=k\pi$ for some $k\in \mathbb{N}^+$.

    Concerning the \textbf{(DC)}$_T$-stabilizability, we have the following result:
\begin{theorem}\label{yu-theorem-7-1-1}
    The system (\ref{yu-6-22-1}) is \textbf{(DC)}$_T$-stabilizable if and only if
    $T\notin\{k\pi:k\in\mathbb{N}^+\}$ (i.e., $T$ is not pathological
   (w.r.t. $A$)).
\end{theorem}
\begin{proof}
    One can easily check that, for each $t\in\mathbb{R}^+$,
\begin{equation*}\label{yu-7-1-1}
\begin{cases}
    \|e^{A^\top t}(\varphi_1, \varphi_2)^\top\|_{\mathbb{R}^2}=\|(\varphi_1, \varphi_2)^\top\|_{\mathbb{R}^2},\\
    B^\top e^{A^\top t}(\varphi_1, \varphi_2)^\top =\varphi_1\sin t+\varphi_2\cos t,
\end{cases}
\;\;\mbox{
    when}\;\;(\varphi_1, \varphi_2)^\top\in\mathbb{R}^2.
\end{equation*}
     Then, according to Theorem \ref{th03}, the system (\ref{yu-6-22-1}) is \textbf{(DC)}$_T$-stabilizable if and only if there is an $N\in\mathbb{N}^+$ and a $C>0$ such that
\begin{equation}\label{yu-7-1-2}
    \|(\varphi_1, \varphi_2)^\top\|^2_{\mathbb{R}^2}
    \leq C\sum_{j=1}^N\Big|\int_{(j-1)T}^{jT}(\varphi_1\sin t+\varphi_2\cos t)dt\Big|^2\;\;\mbox{for any}\;\;(\varphi_1, \varphi_2)^\top\in\mathbb{R}^2.
\end{equation}
\par
   The rest of the proof is organized in two steps.

   \vskip 5pt
   \noindent {\it Step 1. We prove that if the system (\ref{yu-6-22-1}) is \textbf{(DC)}$_T$-stabilizable, then
    $T\notin\{k\pi:k\in\mathbb{N}^+\}$.}

   By contradiction, we suppose that
       the system (\ref{yu-6-22-1}) is \textbf{(DC)}$_T$-stabilizable for some $T\in \{k\pi:k\in\mathbb{N}^+\}$.
      Then, we let $(\varphi_1,\varphi_2)^\top:=(0,\gamma)^\top$ ($\gamma>0$), from which it follows that $\|(\varphi_1,\varphi_2)^\top\|_{\mathbb{R}^2}=\gamma$ and that for each $N\in\mathbb{N}^+$,
\begin{equation*}\label{yu-7-1-3}
    \sum_{j=1}^N\Big|\int_{(j-1)T}^{jT}(\varphi_1\sin t+\varphi_2\cos t)dt\Big|^2
    =\gamma^2\sum_{j=1}^N|\sin(jT)-\sin((j-1)T)|^{2}=0.
\end{equation*}
    These lead to a contradiction of the inequality (\ref{yu-7-1-2}).
\vskip 5pt
   \noindent {\it Step 2. We prove that if $T\notin\{k\pi:k\in\mathbb{N}^+\}$, then the system (\ref{yu-6-22-1}) is \textbf{(DC)}$_T$-stabilizable.}

   We arbitrarily take $T\notin\{k\pi:k\in\mathbb{N}^+\}$, and we denote $a_{1j}=\int_{(j-1)T}^{jT}\sin tdt, a_{2j}=\int_{(j-1)T}^{jT}\cos tdt, j=1, 2$, and  $\Lambda:= (a_{ij})_{1\leq i, j\leq 2}$. Then, the inequality (\ref{yu-7-1-2}) with $N=2$ can be equivalently written as
   \begin{equation*}\label{liu-8-6-1}
    \|(\varphi_1, \varphi_2)^\top\|^2_{\mathbb{R}^2}
    \leq C(\varphi_1, \varphi_2)\Lambda\Lambda^\top(\varphi_1, \varphi_2)^\top.
\end{equation*}
To show the system (\ref{yu-6-22-1}) is \textbf{(DC)}$_T$-stabilizable, it suffices to prove that the matrix $\Lambda$ is non-singular. By direct calculation, the determinant of $\Lambda$ can be obtained as
\begin{equation*}\label{liu-8-6-2}
 \det (\Lambda)=-2\sin T(1-\cos T).
\end{equation*}
Since $T\notin\{k\pi:k\in\mathbb{N}^+\}$, we see that $\det (\Lambda)\neq 0$.

\par
   Hence, we have completed the proof of Theorem \ref{yu-theorem-7-1-1}.
   \end{proof}

   \begin{remark}\label{remark5.2-7-8}
   For a general finite-dimensional setting where $Y:=\mathbb{R}^n$, $U:=\mathbb{R}^m$, $A\in\mathbb{R}^{n\times n}$, and $B\in\mathbb{R}^{m\times n}$ ($n,m\in\mathbb{N}^+$), the connections between the ``pathological'' characteristic, \textbf{(CC)}-stabilizability, and  \textbf{(DC)}$_T$-stabilizability
   have been studied previously \cite[Theorem 8]{Hutus-1970} (see also \cite[Section 3.2]{Chen-96}), via different approaches.
   We believe that results obtained therein can also be verified by our Theorem \ref{th03}.
         \end{remark}

\noindent \textbf{Example 2. }\emph{(Fractional heat equations.)}
 We consider the control fractional heat equation:
  \begin{equation}\label{ex01-1}
\partial_t y(x,t)+(-\triangle)^{\frac{s}{2}}y(x,t)-c y(x,t)=\chi_E(x)u(x,t)\;\; \mbox{in}\;\; \mathbb{R}^n\times\mathbb{R}^+,
\end{equation}
where $c\geq 0$, $s>1$, $E$ is  a  measurable subset  of $\mathbb{R}^n$,  $u\in L^2(\mathbb{R}^+;L^2(\mathbb{R}^n))$, and
the fractional Laplacian $(-\triangle)^{\frac{s}{2}}$ is defined by
$$
(-\triangle)^{\frac{s}{2}}\varphi:=\mathcal{F}^{-1}[|\xi|^s\mathcal{F}[\varphi]],\;\; \varphi\in C_c^\infty(\mathbb{R}^n).
$$
Two points deserve to be mentioned. First, the system  (\ref{ex01-1}) can be put into our framework (\ref{yu-5-26-1}) by setting
 $Y=U:=L^2(\mathbb{R}^n)$,   $A:=-(-\triangle)^{\frac{s}{2}}+c$ with
$D(A)=H^s(\mathbb{R}^n)$, and  $Bv:=\chi_E v$ for each $v\in L^2(\mathbb{R}^n)$.
Second, the system  (\ref{ex01-1}) is not exponentially stable when $u=0$ (this can be directly checked.)

   Concerning the connection between the \textbf{(DC)}$_T$-stabilizability and the \textbf{(CC)}-stabilizability
    for the system (\ref{ex01-1}), we have the following result.
    \begin{theorem}\label{th301}
    The following statements are equivalent:
\begin{enumerate}
  \item [(i)] The system (\ref{ex01-1}) is \textbf{(DC)}$_T$-stabilizable for some $T>0$.
  \item [(ii)] The system (\ref{ex01-1}) is \textbf{(DC)}$_T$-stabilizable for all $T>0$.
  \item [(iii)] The system (\ref{ex01-1}) is \textbf{(CC)}-stabilizable.
  \item [(iv)] The system (\ref{ex01-1}) is completely stabilizable.
  \end{enumerate}
    \end{theorem}

    Theorem \ref{th301} follows immediately from
    \cite[Theorem 1.1]{Huang-2021}, \cite[Theorem 4.5]{Liu-22},
    and the following lemma.

\begin{lemma}\label{th301w7-7}
The following statements are equivalent:
\begin{enumerate}
  \item [(i)] The set $E$ is thick{\footnote{By a thick set $E$ in $\mathbb{R}^n$, we mean that it is measurable and
      satisfies for some constants  $\gamma>0$ and $L>0$,
$$|E\cap Q_L(x)|\geq \gamma L^n \ \mathrm{for\ each}\  x\in \mathbb{R}^n,$$
   where $Q_L(x)$ is the closed cube in $\mathbb{R}^n$ (centered at $x$ and with side-length $L$) and $|E\cap Q_L(x)|$ denotes the Lebesgue measure of $E\cap Q_L(x)$.}} in $\mathbb{R}^n$.
    \item [(ii)]   The system (\ref{ex01-1}) is \textbf{(DC)}$_T$-stabilizable for some $T>0$.

  \item [(iii)] The system (\ref{ex01-1}) is \textbf{(DC)}$_T$-stabilizable for all $T>0$.
\end{enumerate}
   \end{lemma}
\begin{proof}
We organize the proof in several steps.
\vskip 5pt
\noindent {\it Step 1. It is trivial that $(iii)\Rightarrow (ii)$.}

\vskip 5pt

\noindent {\it Step 2. We show that $(ii)\Rightarrow (i)$.}

      Suppose that $(ii)$ is true. Then, according to  $(i)$ in Theorem \ref{thm1.11-7-26},  (\ref{ex01-1}) is \textbf{(CC)}-stabilizable. This, along with \cite[Theorem 1.1]{Huang-2021}, yields that $E$ is a thick set in $\mathbb{R}^n$, i.e., $(i)$ holds.

\vskip 5pt

\noindent {\it Step 3. We show that $(i)\Rightarrow (iii)$.}

Suppose that $(i)$ is true. Then, according to
  \cite[Theorem 1.2]{Wang-Zhang-2021}, for each $r>0$, there is a $C_0(E,  r)>0$ and a $\theta\in(0,1)$ such that
\begin{equation}\label{ex02}
\|S(r)^*\psi\|_{L^2(\mathbb{R}^n)}^2\leq C_0(E,  r)  \|\chi_ES(r)^*\psi\|_{L^2(\mathbb{R}^n)}^{2\theta}\|\psi\|_{L^2(\mathbb{R}^n)}^{2(1-\theta)}\;\;\mbox{for any}\;\;\psi\in L^2(\mathbb{R}^n).
\end{equation}
We arbitrarily fix $T>0$, $\delta\in (0,1)$, $N\in\{2,3,\ldots\}$, and $\varphi\in L^2(\mathbb{R}^n)$. By letting
$\psi:=\int_0^{[N/2]T}S([N/2]T-t)^*\varphi dt$
  and $r:=(N-[N/2])T$ in (\ref{ex02}), we obtain
\begin{eqnarray}\label{ex03}
&\;&\Big\|\int_0^{[N/2]T}S(NT-t)^*\varphi dt\Big\|_{L^2(\mathbb{R}^n)}^2\nonumber\\
&\leq& C(E,N,T)  \Big\|\int_0^{[N/2]T}\chi_ES(NT-t)^*\varphi dt\Big\|_{L^2(\mathbb{R}^n)}^{2\theta}\Big\|\int_0^{[N/2]T}S([N/2]T-t)^*\varphi dt\Big\|_{L^2(\mathbb{R}^n)}^{2(1-\theta)},
\end{eqnarray}
    where $C(E,N,T):=C_0(E,(N-[N/2])T)$.
    Meanwhile, we set  $z(t):=S(t)^*\varphi$ and $t\geq 0$.
     Then, $z(\cdot)$ satisfies the equation (\ref{ex01-1}) (where $u=0$) with the initial condition $y(0)=\varphi$ (here, we use the fact that $(A,D(A))$ is self-adjoint). Using the Fourier transform
     of equation (\ref{ex01-1}) (where $u=0$ and $y$ is replaced by $z$) leads to
\begin{equation}\label{ex04}
\mathcal{F}[z(t)]=e^{(c-|\xi|^s)t}\mathcal{F}[\varphi],\;\;\;t\in\mathbb{R}^+.
\end{equation}
Applying the Plancherel theorem to (\ref{ex04}) yields the following two inequalities:
\begin{eqnarray*}
&\;&\Big\|\int_0^{[N/2]T}S(NT-t)^*\varphi dt\Big\|_{L^2(\mathbb{R}^n)}\nonumber\\
&=&\Big\|\mathcal{F}\Big[\int_0^{[N/2]T}S(NT-t)^*\varphi dt\Big]\Big\|_{L^2(\mathbb{R}^n)}
=\Big\|\int_0^{[N/2]T}\mathcal{F}[z(NT-t)]dt\Big\|_{L^2(\mathbb{R}^n)}\\
&=&\Big\|\Big|\int_0^{[N/2]T}e^{-(c-|\xi|^s)t}dt\Big|\Big|\mathcal{F}[z(NT)]\Big|
\Big\|_{L^2(\mathbb{R}^n)}
\geq [N/2]Te^{-c[N/2]T}\|z(NT)\|_{L^2(\mathbb{R}^n)}
\end{eqnarray*}
and
\begin{eqnarray*}
&\;&\Big\|\int_0^{[N/2]T}S([N/2]T-t)^*\varphi dt\Big\|_{L^2(\mathbb{R}^n)}= \Big\|\int_0^{[N/2]T}\mathcal{F}[z([N/2]T-t)]dt\Big\|_{L^2(\mathbb{R}^n)}\\
&=&\Big\|\int_0^{[N/2]T}e^{(c-|\xi|^s)([N/2]T-t)}dt\mathcal{F}[\varphi]\Big\|_{L^2(\mathbb{R}^n)}
\leq [N/2]T e^{c[N/2]T}\|\varphi\|_{L^2(\mathbb{R}^n)}.
\end{eqnarray*}
The above two inequalities, together with (\ref{ex03}), imply that
\begin{eqnarray*}
\|S(NT)^*\varphi\|_{L^2(\mathbb{R}^n)}^2\leq C(E,N,T,\theta) \Big\|\int_0^{[N/2]T}\chi_ES(NT-t)^*\varphi dt\Big\|_{L^2(\mathbb{R}^n)}^{2\theta}\|\varphi\|_{L^2(\mathbb{R}^n)}^{2(1-\theta)},
\end{eqnarray*}
    where $C(E,N,T,\theta):=\frac{C(E,N,T)}{([N/2]T)^{2\theta}}e^{2c(2-\theta)[N/2]T}$.
Then, using Young's inequality in the above inequalities yields
\begin{equation}\label{ex05}
\|S(NT)^*\varphi\|_{L^2(\mathbb{R}^n)}^2\leq \theta [C(E,N,T,\theta)]^{1/\theta}\delta^{\frac{\theta-1}{\theta}}
\Big\|\int_0^{[N/2]T}\chi_ES(NT-t)^*\varphi dt\Big\|_{L^2(\mathbb{R}^n)}^{2}+\delta\|\varphi\|_{L^2(\mathbb{R}^n)}^{2}.
\end{equation}
Since
\begin{eqnarray*}\Big\|\int_0^{[N/2]T}\chi_ES(NT-t)^*\varphi dt\Big\|_{L^2(\mathbb{R}^n)}^{2}&\leq & [N/2] \sum_{i=1}^{[N/2]}\Big\|\int_{(i-1)T}^{iT}\chi_ES(NT-t)^*\varphi dt\Big\|_{L^2(\mathbb{R}^n)}^{2}\\
&\leq& [N/2]\sum_{i=1}^{N}\Big\|\int_{(i-1)T}^{iT}\chi_ES(NT-t)^*\varphi dt\Big\|_{L^2(\mathbb{R}^n)}^{2},
\end{eqnarray*}
it follows from (\ref{ex05}) that
\begin{eqnarray*}\label{ex06}
&\;&\|S(NT)^*\varphi\|_{L^2(\mathbb{R}^n)}^2\nonumber\\
&\leq& \theta [N/2] [C(E,N,T,\theta)]^{1/\theta}\delta^{\frac{\theta-1}{\theta}}\sum_{i=1}^{N}
\Big\|\int_{(i-1)T}^{iT}\chi_ES(NT-t)^*\varphi dt\Big\|_{L^2(\mathbb{R}^n)}^{2} +\delta\|\varphi\|_{L^2(\mathbb{R}^n)}^{2}.
\end{eqnarray*}
Since $T>0$ was arbitrarily taken, the above leads to
 (\ref{e107}) with any $T>0$.
  Then, by Theorem \ref{th03}, we see that the system (\ref{ex01-1}) is \textbf{(DC)}$_T$-stabilizable for all $T>0$.
\end{proof}
\color{black}

\vskip 5pt \noindent \textbf{Example 3.} (\emph{Schr\"{o}dinger equation}.)
    We consider the control Schr\"{o}dinger equation:
     \begin{equation}\label{ex07}
\verb"i"\partial_ty(x,t)+\partial_{x}^2y(x,t)=u(x,t)\;\;\mbox{in}\;\; \mathbb{R}\times\mathbb{R}^+,
\end{equation}
     where  $u\in L^2(\mathbb{R}^+;L^2(\mathbb{R}; \mathbb{C}))$.
    Several facts are stated as follows. First, equation \eqref{ex07} can be put into our framework (\ref{yu-5-26-1}) by setting  $U=Y:=L^2(\mathbb{R}; \mathbb{C})$, $A:=\verb"i"\partial_{x}^2$ with its domain $D(A)=H^2(\mathbb{R}; \mathbb{C})$,
        $B:=-\verb"i"I$, where $I$ is the identity operator on $L^2(\mathbb{R}; \mathbb{C})$.
         Second, the equation (\ref{ex07}) with the null control
        is not stable. Third, $(A,D(A))$ generates a unitary group $\{S(t)\}_{t\in \mathbb{R}}$ on $L^2(\mathbb{R}; \mathbb{C})$.

    Concerning the \textbf{(CC)}-stabilizability and the \textbf{(DC)}$_T$-stabilizability for
    the system (\ref{ex07}), we have the following result.
 \begin{theorem}\label{th301-b}
  The system (\ref{ex07}) is \textbf{(CC)}-stabilizable but not \textbf{(DC)}$_T$-stabilizable for any $T>0$.
 \end{theorem}
  \begin{proof}
    First, by using the Fourier transform and the Plancherel identity, one can  easily check  that, for each $\gamma>0$, the closed-loop system $\verb"i"\partial_ty(x,t)+\partial_{x}^2y(x,t)=-\verb"i"\gamma y(x,t)$ in $\mathbb{R}\times\mathbb{R}^+$ is exponentially stable. Thus, the system (\ref{ex07}) is \textbf{(CC)}-stabilizable.
\par
    We next prove that the system (\ref{ex07}) is not \textbf{(DC)}$_T$-stabilizable for any $T>0$.
     Since $\{S(t)\}_{{t\in \mathbb{R}}}$ is a unitary group on $L^2(\mathbb{R}; \mathbb{C})$ and $B=-\verb"i"I$,
     according to the equivalence of $(i)$ and $(iii)$ in Theorem \ref{th03},  it suffices to show that for each $T>0$, each $N\in\mathbb{N}^+$, and each $\varepsilon>0$, there is a $\varphi\in L^2(\mathbb{R}; \mathbb{C})$ with $\|\varphi\|_{L^2(\mathbb{R}; \mathbb{C})}=1$ such that
 \begin{equation}\label{yu-6-16-10}
    \sum_{i=1}^N\Big\|\int_{(i-1)T}^{iT}S(t)^*\varphi dt\Big\|_{L^2(\mathbb{R}; \mathbb{C})}\leq \varepsilon.
 \end{equation}
    To prove \eqref{yu-6-16-10}, we arbitrarily fix a $T>0$, $N\in\mathbb{N}^+$, and $\varepsilon>0$.
    Since
    \begin{equation*}
    1-e^{\verb"i"s}=\verb"i"\int_s^{2\pi}e^{\verb"i"\theta}d\theta,
    \end{equation*}
    we see that   for each
    $\eta>0$,
\begin{equation}\label{yu-6-16-11}
    |1-e^{\verb"i"s}|\leq |2\pi-s|\leq \eta\;\;\mbox{when}\;\;s\in(2\pi-\eta,2\pi+\eta).
\end{equation}
    We arbitrarily fix an $\eta\in(0,2\pi)$ such that
\begin{equation}\label{yu-6-16-11-bb}
    \left(\frac{\eta T}{2\pi-\eta}\right)^2\leq \frac{\varepsilon}{N}.
\end{equation}
  Then, we arbitrarily fix a non-zero function $f\in C_c^{\infty}(\mathbb{R}; \mathbb{C})$ such that
\begin{equation}\label{yu-6-16-12}
    \mbox{supp}(f)\subset \mathcal{I}:=\Big(\sqrt{\frac{2\pi-\eta}{T}}, \sqrt{\frac{2\pi+\eta}{T}} \Big)(\subset\mathbb{R}^+),
\end{equation}
    where $\mbox{supp}(f)$ denotes the support of $f$. Let
\begin{equation}\label{yu-6-16-13}
    \varphi:=\|f\|^{-1}_{L^2(\mathbb{R}; \mathbb{C})}\mathcal{F}^{-1}[f].
\end{equation}
   Then, by the Plancherel identity and the fact that $f\neq 0$, we obtain
\begin{equation}\label{yu-6-16-14}
    \|\varphi\|_{L^2(\mathbb{R}; \mathbb{C})}=1.
\end{equation}
    Let
\begin{equation}\label{yu-6-16-14-b}
    z(x,t):=[S(t)^*\varphi](x),\;\; (x,t)\in\mathbb{R}\times\mathbb{R}^+.
\end{equation}
    One can easily check that $z(\cdot,\cdot)$ is the solution to the following equation:
\begin{equation*}\label{yu-6-16-15}
\begin{cases}
    \verb"i"\partial_{t}z(x,t)-\partial_{x}^2z(x,t)=0\;\;\mbox{in}\;\; \mathbb{R}\times\mathbb{R}^+,\\
    z(x,0)=\varphi(x)\;\;\mbox{in}\;\;\mathbb{R}.
\end{cases}
\end{equation*}
    As $\varphi\in L^2(\mathbb{R}; \mathbb{C})$, applying the Fourier transform to the above equation leads to
\begin{equation*}\label{yu-6-16-16}
    \mathcal{F}[z(\cdot,t)](\xi)=e^{\verb"i"\xi^2t}\mathcal{F}[\varphi](\xi),
    \;\;\xi\in\mathbb{R}.
\end{equation*}
    This, together with the fact that $\{S(t)\}_{{t\in \mathbb{R}}}$ is a unitary group, the Plancherel identity, (\ref{yu-6-16-14-b}), (\ref{yu-6-16-13}), and (\ref{yu-6-16-12}) yields that for each $i\in\{1,2,\ldots, N\}$,
\begin{eqnarray}\label{yu-6-16-17}
    &\;&\Big\|\int_{(i-1)T}^{iT}S(t)^*\varphi dt\Big\|^2_{L^2(\mathbb{R}; \mathbb{C})}=\Big\|\int_0^TS(t)^*\varphi dt\Big\|^2_{L^2(\mathbb{R}; \mathbb{C})}
    =\Big\|\int_0^Te^{\verb"i"\xi^2t}\mathcal{F}[\varphi]dt\Big\|^2_{L^2(\mathbb{R}; \mathbb{C})}\nonumber\\
    &\leq&\|f\|^{-2}_{L^2(\mathbb{R}; \mathbb{C})}\int_{\mathcal{I}}\Big|\int_0^T
    e^{\verb"i"\xi^2t}dtf(\xi)\Big|^2d\xi=\|f\|^{-2}_{L^2(\mathbb{R}; \mathbb{C})}\int_{\mathcal{I}}
    \frac{|1-e^{\verb"i"\xi^2T}|^2}{|\xi|^4}|f(\xi)|^2d\xi.
\end{eqnarray}
   Meanwhile, by (\ref{yu-6-16-11}) and (\ref{yu-6-16-12}), we have
\begin{equation*}\label{yu-6-16-18}
    \frac{|1-e^{\verb"i"\xi^2T}|^2}{|\xi|^4}\leq  \left(\frac{\eta T}{2\pi-\eta}\right)^2
    \;\;\mbox{for any}\;\;\xi \in\mathcal{I}.
\end{equation*}
    This, along with (\ref{yu-6-16-17}) and (\ref{yu-6-16-11-bb}), indicates that
\begin{equation*}\label{yu-6-16-19}
    \sum_{i=1}^N\Big\|\int_{(i-1)T}^{iT}S(t)^*\varphi dt\Big\|_{L^2(\mathbb{R}; \mathbb{C})}
    \leq N\|f\|^{-2}_{L^2(\mathbb{R}; \mathbb{C})}\Big(\frac{\eta T}{ 2\pi-\eta}\Big)^2
    \int_{\mathcal{I}}|f(\xi)|^2d\xi\leq N\Big(\frac{\eta T}{2\pi-\eta}\Big)^2\leq \varepsilon,
\end{equation*}
      which, together with (\ref{yu-6-16-14}), shows that $\varphi$ defined by
      (\ref{yu-6-16-13}) verifies (\ref{yu-6-16-10}).
\par
    Thus, we have completed the proof of Theorem \ref{th301-b}.
 \end{proof}

\bibliographystyle{plain}

\end{document}